\documentclass[reqno]{amsart}
\usepackage{amsmath,amssymb,mathrsfs,amsthm,amsfonts}
\usepackage[inline]{enumitem} 
\usepackage[usenames,dvipsnames]{xcolor}
\usepackage{hyperref}
\hypersetup{%
	colorlinks=true, linkcolor=ForestGreen,
	citecolor=ForestGreen
}
\usepackage[paper=letterpaper,margin=1in]{geometry}

\newtheorem{theorem}{Theorem}[section]
\newtheorem{innercustomthm}{Theorem}
\newenvironment{customthm}[1]
{\renewcommand\theinnercustomthm{#1*}\innercustomthm}
{\endinnercustomthm}
\newtheorem{lemma}[theorem]{Lemma}
\newtheorem{proposition}[theorem]{Proposition}

\theoremstyle{definition}

\newtheorem{remark}[theorem]{Remark}
\numberwithin{equation}{section}
\allowdisplaybreaks
\usepackage{acronym}

\acrodef{KPZ}{Kardar--Parisi--Zhang}
\acrodef{SHE}{Stochastic Heat Equation}
\acrodef{LDP}{Large Deviation Principle}

\renewcommand{\Pr}{\mathbf{P}}	
\newcommand{\Ex}{\mathbf{E}}	
\renewcommand{\d}{\mathrm{d}}	
\newcommand{\ind}{\mathbf{1}}	
\newcommand{\set}[1]{{\{#1\}}}	
\newcommand{\sgn}{\mathrm{sign}}
\newcommand{\tr}{\mathrm{tr}} 	
\newcommand{\normm}[1]{|#1|}

\newcommand{\norm}[1]{\Vert#1\Vert}
\newcommand{\Norm}[1]{\big\Vert#1\big\Vert}

\newcommand{\hk}{p}					
\newcommand{\ai}{\operatorname{Ai}}	
\newcommand{\aisq}{\Phi}			
\newcommand{\expf}{U}				

\newcommand{\Con}{\mathrm{C}} 

\newcommand{\R}{\mathbb{R}} 
\newcommand{\Z}{\mathbb{Z}} 
\newcommand{\Msp}{\mathfrak{M}} 

\newcommand{\e}{\varepsilon}
\newcommand{\op}{\mathrm{op}}
\newcommand{\s}{\sigma}
\newcommand{\calA}{\mathcal{A}}
\newcommand{\calB}{\mathcal{B}}
\newcommand{\calH}{\mathcal{H}}

\newcommand{\calI}{\mathcal{I}}
\newcommand{\calJ}{\mathcal{J}}
\newcommand{\calZ}{\mathcal{Z}}

\renewcommand{\hat}{\widehat}
\newcommand{\til}{\widetilde}
\renewcommand{\bar}{\overline}

\usepackage{graphicx}
\newcommand*{\Cdot}{{\raisebox{-0.5ex}{\scalebox{1.8}{$\cdot$}}}} 

\newcommand{\revi}{\textcolor{black}}

\title[Fractional Moments of the SHE]{Fractional Moments of the Stochastic Heat Equation}

\author[S.\ Das]{Sayan Das}
\address{S.\ Das,
	Departments of Mathematics, Columbia University,
	\newline\hphantom{\quad \ \ S. Das}
	2990 Broadway, New York, NY 10027 USA
	}
\email{sayan.das@columbia.edu}
\author[L.-C.\ Tsai]{Li-Cheng Tsai}
\address{L.-C.\ Tsai,
	Department of Mathematics, Rutgers University --- New Brunswick
	\newline\hphantom{\quad \ \ L.-C. Tsai}
	10 Frelinghuysen Road, Piscataway, NJ 08854 USA
}
\email{lctsai.math@gmail.com}
\begin{document}
\begin{abstract}
Consider the solution $\mathcal{Z}(t,x)$ of the one-dimensional stochastic heat equation, with a multiplicative spacetime white noise, 
and with the delta initial data $ Z(0,x) = \delta(x) $.
For any real $ p>0 $, we obtained detailed estimates of the $ p $-th moment of $ e^{t/12}\mathcal{Z}(2t,0) $, as $t\to\infty$,
and from these estimates establish the one-point upper-tail large deviation principle of the Kardar--Parisi--Zhang equation.
The deviations have speed $ t $ and rate function $ \Phi_+(y)=\frac43y^{3/2} $.
Our result confirms the existing physics predictions \cite{ledoussal16long} and also \cite{kamenev16}.
\end{abstract}

\maketitle

\section{Introduction}
In this article we study the \ac{SHE} in one spatial dimension
\begin{align}\label{she}
	\hspace{.2\linewidth}
	\partial_t\mathcal{Z} = \tfrac12\partial_{xx}\mathcal{Z} + \xi\mathcal{Z}, 
	\qquad \calZ=\calZ(t,x), \quad (t,x)\in [0,\infty)\times\R,
\end{align}
where $ \xi=\xi(t,x) $ is the Gaussian spacetime white noise.
Via the Feynman--Kac formula, solutions of the \ac{SHE} gives the partition function of the directed polymer in a continuum random environment \cite{huse1985huse,comets2005directed}.
On the other hand, the logarithm $ \mathcal{H}(t,x) := \log \mathcal{Z}(t,x) $ formally solves the \ac{KPZ} equation
\begin{align}\label{kpz}
	\partial_t\mathcal{H}=\tfrac12\partial_{xx}\mathcal{H} + \tfrac12(\partial_x \mathcal{H})^2 + \xi.
\end{align} 
Introduced in~\cite{kardar1986dynamic}, the \ac{KPZ} equation is 
\revi{a paradigm} for random surface growth.
It connects to many physical systems including \revi{directed polymers}, last passage percolation, random fluids, interacting particle systems, 
and exhibits statistical behaviors similar to certain random matrices.
We refer to \cite{ferrari2010random,quastel2011introduction,corwin2012kardar,quastel2015one,chandra2017stochastic,corwin2019} and the references therein for the mathematical study of and related to the \ac{KPZ} equation.

Throughout this paper we will consider the solution $ \calZ(t,x) $ of the \ac{SHE}~\eqref{she} with the initial data
\begin{align}
	\label{deltaic}
	\calZ(0,x) = \delta(x),
\end{align}
the Dirac delta function at the origin.
The \ac{SHE}~\eqref{she} enjoys a well-developed solution theory based on It\^{o} integral and chaos expansion~\cite{walsh1986,bertini1995stochastic}, also \cite{quastel2011introduction,corwin2012kardar}.
In particular, there exists a unique $ C((0,\infty),\R)$-valued process $ \calZ $ that solves \eqref{she} with the delta initial data~\eqref{deltaic} in the mild sense, i.e.,
\begin{align*}
	\calZ(t,x) = \hk(t,x) + \int_0^t\int_{\R} \hk(t-s,x-y) \calZ(s,y) \xi(s,y) \, \d s \d y,
\end{align*}
where $ \hk(t,x) := (2\pi t)^{-1/2} \exp(-x^2/(2t)) $ denotes the standard heat kernel.

The solution $ \calZ $ of the \ac{SHE} can be transformed into a solution of the \ac{KPZ} equation.
\revi{%
For a nonzero initial data $ \calZ(0,\Cdot) $ that is bounded, nonnegative, and has a compact support, \cite{mueller1991support} showed that almost surely $ \calZ(t,x) > 0 $ for all $ (t,x)\in(0,\infty)\times\R $. For the delta initial data \eqref{deltaic} considered here, the same positivity result was established in \cite{flores14}.%
}
The logarithm $ \calH(t,x) := \log\calZ(t,x) $ is defined to be \textbf{Hopf--Cole solution} of the \ac{KPZ} equation. 
This is the notion of solutions that we will be working with throughout this paper.
The motivation is, as mentioned previously, that non-rigorously taking logarithm in~\eqref{she} yields the \ac{KPZ} equation~\eqref{kpz}. 
The \ac{KPZ} equation~\eqref{kpz} itself is ill-posed due to the roughness of the solution and the presence of the quadratic term.
New theories have been developed for making sense of  \revi{the} \ac{KPZ} equation and constructing the corresponding solution process.
This includes \revi{regularity structures} \cite{hairer2013solving,hairer14}, paracontrolled distributions \cite{gubinelli2015paracontrolled,gubinelli17}, and \revi{energy solutions}~\cite{gonccalves2014nonlinear,gubinelli2018}.
The Hopf--Cole formulation bypasses the ill-posedness issue, and arises in several discrete or regularized version of the \ac{KPZ} equation, e.g., \cite{bertini1995stochastic, bertini1997stochastic}. Further, other notions of solutions from the aforementioned theories have been shown to coincide with the Hopf--Cole solution within the class of initial datas the theory applies.

Of interest is the large time behaviors of $ \calH(t,x) := \log\calZ(t,x) $.
Simultaneously and independently, the physics works \cite{calabrese10,dotsenko10,sasamoto10} and mathematics work \cite{amir10}
gave the following large $ t $ asymptotic fluctuation result of $ \calH(t,x) $,
and \cite{amir10} provided a rigorous proof:
\begin{align*}
	\tfrac{1}{t^{1/3}} \big( \calH(2t,0)+\tfrac{t}{12} \big)
	\Longrightarrow
	\text{GUE Tracy--Widom distribution}.
\end{align*}
This result asserts that, for large $ t $, the height $ \calH(2t,0) $ concentrates around $ -\frac{t}{12} $, has typical deviations of order $ t^{1/3} $, and after being scaled by $ t^{-1/3} $ the fluctuations converge to the GUE Tracy--Widom distribution \cite{tracy94}.

A natural question that follows the fluctuation result is establishing a \ac{LDP}, namely questions about \emph{tails} of the distribution of $ \calH(2t,0)+\tfrac{t}{12} $. We seek to find the probability of the rare events when the height $\calH(2t,0)+\tfrac{t}{12}$ has a deviation of order $t$. Interestingly the lower- and upper-tail \acp{LDP} have different speeds. The lower-tail deviations occurs at speed $t^2$ while the upper-tail deviations occurs at speed $t$. 
\begin{align}
\label{e.low}
& \Pr[\calH(2t,0)+\tfrac{t}{12}<ty] \approx e^{-t^2\Phi_-(y)}, & (y<0)
\\
\label{e.up}
& \Pr[\calH(2t,0)+\tfrac{t}{12}>ty] \approx e^{-t\Phi_+(y)}. & (y>0)
\end{align}
\noindent Such distinct speeds \revi{can be heuristically explained by} directed polymers. For a \emph{discrete} polymer on an $N\times N$ grid with i.i.d.~site weights, 
 we consider the point to point partition function. 
 It can be made anomalously large by increasing the weights along any \emph{single} path. The cost of changing the weights of $N$ such sites amounts to $ \revi{\exp}(-\mathcal{O}(N)) $. However, smaller partition function can be realized only when the weights along \emph{most} of the paths are decreased jointly. This \revi{can} occur with probability $ \exp(-\mathcal{O}(N^2)) $ \revi{by decreasing} the weights of most of the sites, \revi{c.f., Remark~\ref{r.speed}}.
 For the KPZ equation, recall that the Feynman--Kac formula identifies solution of the SHE as the partition function of the directed polymer in a continuum random environment. 
 This is analogous to discrete polymers, with Brownian motion replacing random walks and space-time white noise replacing site weights.
 In the continuum setting $t$ plays the analogous role as $N$, since both $t$ and $N$ parametrize the polymer length.  
 Identifying $t$ with $N$, we should expect the $t^2$ vs $t$ speeds in \eqref{e.low} and \eqref{e.up}. These speeds were predicted in the physics work  \cite{ledoussal16long}, where the prescribed polymer argument was given. %

\begin{remark}\label{r.speed}
\revi{%
The speed of lower-tail deviations is in fact not universal when the random environment is unbounded.
Specifically, \cite{benari09} showed that the lower-tail speed of the directed polymer with a Gaussian environment is $ N^2/\log N $ instead of $ N^2 $.%
}
\end{remark}
 
Recently there has been much development around the large deviations of the \ac{KPZ} equation in the mathematics and physics communities.
Employing the optimal fluctuation theory, the physics works \cite{kolokolov07,kolokolov09,meerson16} predicted various tail behaviors of the \ac{KPZ} equation.
These predictions were further supported by the analysis of exact formulae in the physics works~\cite{ledoussal16short,krajenbrink17short,krajenbrink18simple}.
In mathematics terms, the optimal fluctuation theory corresponds to Fredilin--Wentzell type large deviations of stochastic PDEs with a small noise.
There has been rigorous treatment \cite{hairer15,cerrai19} of such large deviations for certain nonlinear stochastic PDEs.

Under the same initial data as this paper, the physics works \cite{sasorov2017large,corwin2018coulomb,krajenbrink18} each employed a different method to derive the same explicit rate function for the lower-tail deviations of $ \calH(2t,0)+\tfrac{t}{12} $. The work \cite{corwin2018lower} provides detailed, rigorous bounds on tails of $ \calH(2t,0)+\tfrac{t}{12} $,
which are valid for all $ t>0 $ and capture a crossover behavior predicted in~\cite{kolokolov09,meerson16}.
The lower-tail \ac{LDP} with the exact rate function was later proven in \cite{tsai2018exact}, and more recently in \cite{cafasso2019riemann}.
The four different routes \cite{sasorov2017large,corwin2018coulomb,krajenbrink18,tsai2018exact} of deriving the lower-tail \ac{LDP} were later shown to be closely related \cite{krajenbrink19}. Two new routes have been recently obtained in the rigorous work \cite{cafasso2019riemann} and physics work \cite{doussal2019large}.

In this paper we focus on the \emph{upper} tail --- the complement of the aforementioned results.
Since $ \calZ(t,x) = \exp(\calH(t,x)) $, the upper tail is closely related to positive moments of $ \calZ $.
The moments of SHE and its connection to intermittency property \cite{gartner1990parabolic,gartner2007geometric} has been previously studied in \cite{conus2013chaotic, chen2015moments, chen2015precise, khoshnevisan2017intermittency}. 
These works
established finite time estimates of tails or moments of $ \calZ(t,x) $ and solutions of related stochastic PDEs. 
\revi{%
The work \cite{chen16} studied a class of equations that includes the \ac{SHE} with the delta initial data considered here.
With the aim of establishing the existence of the smooth density,
the work obtained finite time tail estimates of the solution.%
}

For the large time regime considered here, the form $ \Phi_+(y)=\frac43 y^{3/2} $ was predicted in \cite{ledoussal16long} by analyzing an exact formula. The analysis also yields subdominant corrections; see \cite[Supp.~Mat.]{doussal2016large}. We note that, for the short time regime, \cite{kamenev16} predicted the same $\frac32$-power law. A priori, the optimal fluctuation theory used therein works only for short time, although the validity in large time was argued therein. For the large time regime, \cite{corwin2013crossover} gave a bound on of the upper tail of $ \calZ(t,x) $ (with a different initial data). The bound exhibits the predicted $\frac32$-power for small $y$ but not large $y$. 
Extracting information from positive integer moments of $\calZ$, \cite{corwin2018kpz} provided detail bounds on the upper-tail probability. The upper and lower bounds therein capture the aforementioned $\frac32$-power law but do not match as $t\to \infty$. 

In this paper we present the \emph{first rigorous proof} of the upper-tail \ac{LDP} of $ \calH(2t,0)+\tfrac{t}{12} $ with the predicted $\Phi_+(y)=\tfrac43y^{3/2}$ rate function.  Interestingly, this matches exactly with the upper-tail rate function for the Tracy-Widom distribution \cite{tracy94}. Our main result gives both the $ t\to\infty $ asymptotic of the $ p $-th moment of $ \calZ(2t,0) $, for any real $ p>0 $,
and the upper-tail \ac{LDP} of the \ac{KPZ} equation.

\begin{theorem} \label{thm.main} 
Let $ \calZ(t,x) $ be the solution of the \ac{SHE}~\eqref{she} with the delta initial data~\eqref{deltaic},
and let $ \calH(t,x) := \log \calZ(t,x) $ be the Hopf--Cole solution of the \ac{KPZ} equation~\eqref{kpz}.
\begin{enumerate}[label=(\alph*), leftmargin=15pt]
\item \label{thm.main.mom}
For any $ p\in(0,\infty) $, we have
\begin{align}
	\label{e.thm.mom}
	\lim_{t\to\infty} \frac{1}{t} \log \Ex\Big[ e^{p(\calH(2t,0)+\frac{t}{12})} \Big]
	=	
	\lim_{t\to\infty} \frac{1}{t} \log \Ex\Big[ \big(\calZ(2t,0)e^{\frac{t}{12}}\big)^p \Big]
	=
	\frac{p^3}{12}.
\end{align}
\item \label{thm.main.ldp}
For any $ y\in(0,\infty) $, we have
\begin{align}
	\label{e.thm.ldp}
	\lim_{t\to\infty} \frac{1}{t} \log \Pr\Big[ \calH(2t,0)+\tfrac{t}{12} \geq ty \Big]
	=-\Phi_+(y):=	
	-\frac{4}{3}y^{3/2}.
\end{align}
\end{enumerate}
\end{theorem}

\begin{remark}
\revi{%
The results in Theorem~\ref{thm.main} immediately generalize to $ x\neq 0 $. This is so because, under the delta initial data \eqref{deltaic}, the random variables $ \calZ(2t,0) $ and $ \calZ(2t,x)\exp(x^2/4t) $ have the same law. This fact can be verified from either the Feynman--Kac formula or the chaos expansion. Hence, the results in Theorem~\ref{thm.main} hold with $ \calZ(2t,x)\exp(x^2/4t) $ replacing $ \calZ(t,0) $ and $ \calH(2t,x) + \frac{x^2}{4t} $ replacing $ \calH(2t,0) $.%
}
\end{remark}

Our method is based on a perturbative analysis of Fredholm determinants, and the major input is the formula \eqref{e.fred} that expresses the Laplace transform of $ \calZ(2t,0) $ as a Fredholm determinant.
We emphasize that our method \emph{differs} from existing methods used in the same context. 
The work \cite{ledoussal16long} postulates a form of the upper tail and verifies a posteriori the consistency with the formula \eqref{e.fred}; see \cite[Supp.~Mat.]{doussal2016large}. There are, however, infinitely many postulated forms that are consistent with \eqref{e.fred}. We explain this phenomenon in Section~\ref{sect.nonunique}. There we reprint the consistency check as a variational problem~\eqref{e.variational}, which has infinitely many solutions given in~\eqref{e.nonunique}.
The work \cite{corwin2013crossover} utilizes an formula of the tail probabilty of $ \calH(2t,0)+\frac{t}{12} $, under the Brownian initial data. Such a formula can be viewed as the inverse Laplace transform of~\eqref{e.fred}. By analyzing the inverse Laplace transform formula, it was shown \cite[Corollary 14]{corwin2013crossover} that there exists constants $c_1,c_2,c_3$ such that for all $y>0$ and large enough $t$ 
\begin{align*}
	\Pr\left[\calH(2t,0)+\tfrac{t}{12}\ge ty\right] \le c_1t^{1/2}e^{-c_2yt}+c_1t^{1/2}e^{-c_3y^{3/2}t}.
\end{align*}
This bound exhibits the $\frac32$-power law for small $y$ but becomes linear in $ y $ (in the exponent) for large $y$.
In this paper we employ a new way of utilizing the formula \eqref{e.fred}, by applying it for getting the $ p $-moment growth of $ \calZ(2t,0) $.

The main body of our proof is devoted to proving Theorem~\ref{thm.main}\ref{thm.main.mom}, or more precisely its refined version Theorem~\ref{thm.main.} stated in the following. From Theorem~\ref{thm.main}\ref{thm.main.mom} standard argument produces Theorem~\ref{thm.main}\ref{thm.main.ldp}, with the rate function $ -\frac{4}{3}y^{3/2} $ being the Legendre transform of $ \frac{p^3}{12} $.
The first \revi{indication} of Theorem~\ref{thm.main}\ref{thm.main.mom} being true came form the study of positive integer moments of \eqref{e.thm.mom}.
The mixed joint moment of $ \calZ $ solves the delta Bose gas, and the delta Bose operator can be diagonalized by the Bethe ansatz.
The work~\cite{kardar87} carried out such analysis and pointed out that \eqref{e.thm.mom} should hold for positive integers, i.e.,
\begin{align}
	\tag{\ref*{e.thm.mom}-int}
	\label{e.thm.mom.}
	\lim_{t\to\infty} \frac{1}{t} \log \Ex\Big[ e^{n(\calH(2t,0)+\frac{t}{12})} \Big]
	=
	\frac{n^3}{12},
	\quad
	\text{for } n \in \Z_{> 0}.
\end{align}
This assertion~\eqref{e.thm.mom.} was proven in \cite{chen2015precise} for function-valued, bounded initial data, and in \cite[Lemma 4.5]{corwin2018kpz} for the delta initial data considered here.
It has long been speculated and conjectured that~\eqref{e.thm.mom.} should extend to all positive real $ p $. However, the connection to the delta Bose gas only gave access to integer moments. Here, by utilizing a known formula but in an unconventional way, we bridge the gap between integers. In the same spirit as \cite[Lemma 4.5]{corwin2018kpz}, we provide a quantitative bound on the $ p $-th moment of $ \calZ $ that holds for all $ t $ and $ p $ away from $ 0 $. This is stated as a refined version of Theorem~\ref{thm.main}\ref{thm.main.mom} as

\begin{customthm}{{\ref*{thm.main}\ref*{thm.main.mom}}} \label{thm.main.} 
Let $ \calZ $ be as in Theorem~\ref{thm.main}.
We have a decomposition 
\begin{align*}
	\Ex\big[ \big(\calZ(2t,0)e^{\frac{t}{12}}\big)^p \big]
	=
	\calA_p(t) + \calB_p(t)
\end{align*}
of the $ p $-th moment of $ \calZ(2t,0)e^{\frac{t}{12}} $ into a leading term $ \calA_p(t) $ and remainder term $ \calB_p(t) $.
For any $ t_0,p_0>0 $, there exists a constant $ \Con=\Con(t_0,p_0)>0 $ that depends only on $ t_0,p_0 $, such that
for all $ t\geq t_0 $ and $ p\geq p_0 $,
\begin{align}
	\label{e.Aest}
	\frac1{\Con} p^{-\frac32}\Gamma(p+1)\,t^{-\frac12}\, e^{\frac{p^3t}{12}}
	\leq
	&\calA_p(t) \leq \Con p^{-\frac32}\Gamma(p+1)\,t^{-\frac12}\, e^{\frac{p^3t}{12}},
\end{align}
and for $ \revi{n:=\lfloor p\rfloor+1}\in\Z_{>0} $ and $ \kappa_p:=\min\{\frac{1}6,\frac{p^3}{16}\} $,
\begin{align}
\label{e.Bbd}
&|\calB_p(t)| \leq n\cdot(n!)^2 \,(n\Con)^{n} \, t^{\frac12} \, e^{\frac{p^3t}{12}-\kappa_p t}.
\end{align}
\end{customthm}
\noindent
From the bounds~\eqref{e.Aest} and \eqref{e.Bbd}, one see that $ \calA_p(t) $ dominates as $ t\to\infty $,
uniformly over any close intervals in $ (0,\infty)\ni p $.
Theorem~\ref{thm.main.} immediately implies Theorem~\ref{thm.main}\ref{thm.main.mom}.

The upper tail problem has also been studied for several other models in the class of integrable systems starting from the fluctuation results and \ac{LDP} for the longest increasing subsequence \cite{kim1996increasing,seppalainen1998large,deuschel1999increasing,baik1999distribution}. There are also analogous results on upper-tail \ac{LDP} for integrable polymer models \cite{georgiou13,janjigian15}, and also for last passage percolation in Bernoulli and white noise environments \cite{ciech18,janjigian19} and inhomogeneous corner growth models \cite{emrah15}.

The main input of our proof is the known formula \eqref{e.fred} that express the Laplace transform of $ \calZ(2t,0) $ as a Fredholm determinant.
There are multiple equivalent ways to define Fredholm determinants~\cite{simon77}.
We will work with the exterior algebra definition: for a trace-class operator $ T $ on a Hilbert space,
consider $ \bigwedge_{i=1}^L H $ and the operator $ T^{\wedge L} $ defined by $  T^{\wedge L}(v_1\wedge\cdots\wedge v_L) := (Tv_1)\wedge\cdots\wedge (Tv_L) $. The operator $ T^{\wedge L} $ is trace-class on  $ \bigwedge_{i=1}^L H $.
We then define the Fredholm determant as
\begin{align*}
	\det(I-T) := 1 + \sum_{L=1}^\infty (-1)^L \tr(T^{\wedge L}).
\end{align*}
\revi{The following formula is known thanks to the integrability of the \ac{SHE} and related models:}
\begin{align}
	\label{e.fred}
	\Ex\big[ \exp(-s\calZ(2t,0)e^{\frac{t}{12}}\big)\big] 
	= 
	\det(I-K_{s,t})
	=
	1 +\sum_{L=1}^\infty (-1)^L \tr(K_{s,t}^{\wedge L}),
\end{align} 
where $ K_{s,t} $ is an integral operator $ L^2(\R_{\geq 0}) $ with the kernel
\begin{align}
	\label{e.K}
	K_{s,t}(x,y):=\int_{\R} \frac{\ai(x+r)\ai(y+r)}{1+\frac1se^{-t^{1/3}r}}\d r,
\end{align}
and $\ai(x) $ is the Airy function.
It is standard to check that $ K_{s,t} $ is a positive operator via the square-root trick, c.f., Lemma~\ref{l.sqtrick}.
\revi{%
The formula \eqref{e.fred} or its closely related forms was first derived simultaneously and independently in \cite{amir10,calabrese10,dotsenko10,sasamoto10}, with a rigorous proof given in \cite{amir10} based on results of \cite{tracy09}. In particular, the formula \eqref{e.fred} can be obtained by taking Laplace transform of \cite[Eq.\,(1.13)]{amir10}. A direct derivation of \eqref{e.fred} with a rigorous proof can be found in \cite{borodin2014free}; see Theorem~1.10\,(a) and Eq.\,(1.7) therein.%
}

A standard way to extract tail information from~\eqref{e.fred} is to parameterize $ s=e^{-ty} $ and substitute in $ \calZ(2t,0)=\exp(\calH(2t,0)) $ to get
\begin{align}
	\label{e.fred.}
	\Ex\big[ \exp(-e^{\calH(2t,0)+\frac{t}{12}-ty})\big] 
	= 
	1 - \tr(K_{s,t}) + \sum_{L=2}^\infty (-1)^L \tr(K_{s,t}^{\wedge L}). 
\end{align} 
The double exponential function $ \exp(-e^{\Cdot}) $ on the l.h.s.\ of \eqref{e.fred.} may be deemed as a good proxy of the indicator function $ \ind_{(-\infty,0)} $, and hence analyzing the r.h.s.\ of~\eqref{e.fred.} could produce information on $ \Pr[ \calH(2t,0)+\frac{t}{12}<ty ] $. This approximation procedure has been successfully implemented in getting the limiting fluctuations and lower-tail LDP (but using different representations of the r.h.s.~than the Fredholm determinant). 

\subsection{An issue of nonuniqueness}
\label{sect.nonunique}
However, for the upper tail, the preceding procedure would not produce the full \ac{LDP}.
To see this, rewrite~\eqref{e.fred.} as
\begin{align}
	\label{e.explain1}
	\Ex\big[ 1-\exp(-e^{\calH(2t,0)+\frac{t}{12}-ty})\big] 
	= 
	\sum_{L=1}^\infty (-1)^{L-1} \tr(K^{\wedge L}_{e^{-ty},t}).
\end{align} 
For $ y>0 $, it is possible to show that the r.h.s.\ of~\eqref{e.explain1} is dominated by the $ L=1 $ term as $ t\to\infty $, 
and analyzing the trace of $ K_{s,t} $ from the formula~\eqref{e.K} should yield
\begin{align*}
	\lim_{t\to\infty} \frac{1}{t}\log\big( \text{r.h.s.\ of }\eqref{e.explain1} \big) = I(y) 
	:= 
	\left\{\begin{array}{l@{,\quad}l}
		-\frac{4}{3}y^{3/2}& y\in(0,\frac14],
	\\
		\frac{1}{12}-y	& y\in(\frac14,\infty).
	\end{array}\right.
\end{align*}
For the left hand side, if we assume the existence of the upper-tail \ac{LDP} but with an unknown rate function,
i.e., $ \lim_{t\to\infty} \frac{1}{t} \log \Pr[ \calH(2t,0)+\frac{t}{12}>ty ] = - \Phi_+(y) $, for $ y\in(0,\infty) $,
using the fact that $ 1-\exp(-e^{t\xi}) \approx \exp(t\revi{\min\{\xi,0\}}) $, as $ t\to\infty $, we should have
\begin{align*}
	\lim_{t\to\infty} \frac{1}{t} \log \Ex\big[ 1-\exp(-e^{\calH(2t,0)+\frac{t}{12}-ty})\big] 
	=
	\sup_{\xi>0} \big\{ \min\{\xi-y,0\}-\Phi_+(\xi) \big\}. 
\end{align*}
Putting these two sides together suggests the variational problem
\begin{align}
	\label{e.variational}
	\sup_{\xi>0} \big\{ \min\{\xi-y,0\}-\Phi_+(\xi) \big\}
	= 
	\left\{\begin{array}{l@{,\quad}l}
		-\frac{4}{3}y^{3/2}& y\in(0,\frac14],
	\\
		\frac{1}{12}-y	& y\in(\frac14,\infty).
	\end{array}\right.
\end{align}
The function $ \Phi_+(y) = \frac{4}{3}y^{3/2} $ does solve this variational problem. However, the solution is \emph{not} unique. 
\emph{Any} function that satisfies
\begin{align}
	\label{e.nonunique}
	\Phi_+(y)=-\tfrac{4}{3}y^{3/2},\text{ for } y\in(0,\tfrac14],
	\qquad
	\tfrac{1}{12}-y \leq \Phi_+(y) \leq \tfrac{4}{3}y^{3/2},\text{ for } y\in(\tfrac14,\infty)
\end{align}
solves the preceding variational problem.

The preceding calculations strongly suggest that the conventional scheme~\eqref{e.fred.} and \eqref{e.explain1}
of using the Fredholm determinant would not produce the exact rate function.

\subsection{Our solution}
To circumvent the aforementioned issue, we provide a new way of using the formula~\eqref{e.fred}.
The start point is the following elementary identity:
\begin{lemma}\label{l.fractionalmom}
Let $U$ be a nonnegative random variable with a finite $n$-th moment, where $n\in\Z_{> 0}$. Let $\alpha\in [0,1)$ Then the $(n-1+\alpha)$-th  moment of $U$ is given by
\begin{align} \label{fmoment}
	\Ex[U^{n-1+\alpha}]
	=
	\frac1{\Gamma(1-\alpha)}\int_{0}^{\infty} s^{-\alpha}\Ex[U^n e^{-sU}]\,\d s=\frac{(-1)^n}{\Gamma(1-\alpha)}\int_{0}^{\infty} s^{-\alpha}\frac{\d^n~}{\d s^n}\Ex[e^{-sU}]\,\d s.
\end{align}
\end{lemma}
\noindent The proof of this lemma follows by an interchange of measure via Fubini's theorem. 
We will apply this lemma with $U=\calZ(2t,0)e^{\frac{t}{12}}$ and with $n:=\lfloor p \rfloor +1 \in \Z_{> 0}$ and $\alpha:=p-\lfloor p \rfloor \in [0,1)$ so that $p=n-1+\alpha$. 

Utilizing the formula~\eqref{e.fred} for $ \Ex[e^{-sU}]=\Ex[e^{-s\calZ(2t,0)e^{\frac{t}{12}}}] $ in~\eqref{fmoment},
we will then be able to express the $ p $-th moment of $ \calZ(2t,0)e^{\frac{t}{12}} $ as a series.
From this series we identify the leading term and higher order terms. 
This eventually leads to the desired estimate in Theorem~\ref{thm.main.}.

\revi{%
It seems possible to directly analyze the inverse Laplace transform formula in \cite[Theorem 1.1]{amir10}.
Doing so may provide an alternative proof of Theorem~\ref{thm.main}\ref{thm.main.ldp}.}

\subsection*{Outline.} 
In Section~\ref{sec.prelim} we setup the framework of the proof.
Namely we introduce an expansion of the $ p $-th moment of $ \calZ $, identify a trace term as the leading term, and establish several technical lemmas.
In Section~\ref{sec.trace}, we give precise asymptotics of the leading trace term,
and in Section~\ref{sec.higher} we establish bounds on the remaining terms.
Finally, in Section~\ref{sec.pfthm}, we collect results from previous sections to give a proof of Theorem~\ref{thm.main} and Theorem~\ref{thm.main.}.

\subsection*{Acknowledgements.} 
We thank Ivan Corwin for suggesting the problem and giving us useful inputs in an earlier draft of the paper. We thank Promit Ghosal and Shalin Parekh for helpful conversations and discussions. 
We thank Chris Janjigian and Pierre Le Doussal for useful comments on improving the presentation of this paper.
\revi{%
We thank the anonymous referees for their careful reading and useful comments on improving our manuscript. 
The phenomenon stated in Remark~\ref{r.speed} was pointed out to us by a referee during the reviewing process.%
}

SD's research was partially supported from Ivan Corwin’s NSF grant DMS-1811143.
LCT's research was partially supported by the NSF through DMS-1712575.

\section{Basic framework} \label{sec.prelim}
Throughout this paper we use $\Con=\Con(a,b,c,\ldots)>0$ to denote a generic deterministic positive finite constant 
that may change from line to line, but dependent on the designated variables $a,b,c,\ldots$. 

As mentioned previously, we will utilize Lemma~\ref{l.fractionalmom} and \eqref{e.fred} to develop a series expansion for $ \Ex[(\calZ(2t,0)e^{\frac{t}{12}})^p] $.
This, however, requires a truncation at $ s=1 $ first.
To see why, referring to~\eqref{e.fred.}, with $ s=e^{-ty} $, we see that $ s<1 $ corresponds to upper tail while $ s>1 $ corresponds to lower tail.
While we expect the later to have minor contribution in the regime $ p>0 $ we are probing, it is known that for $ s\gg 1 $ the Fredholm determinant \eqref{e.fred.} behaves \revi{in an oscillatory fashion as $t\to\infty$}.
With  $n:=\lfloor p \rfloor +1 \in \Z_{> 0}$ and $\alpha:=p-\lfloor p \rfloor \in [0,1)$, we truncate
\begin{align} \label{e.fmoment}
	\Ex\big[(\calZ(2t,0)e^{\frac{t}{12}})^p\big]
	=
	\frac{(-1)^n}{\Gamma(1-\alpha)}\int_{0}^{1} s^{-\alpha} \, \partial^n_s \, \Ex[e^{-s\calZ(2t,0)e^{\frac{t}{12}}}]\,\d s
	+
	\calB_{p,1}(t),
\end{align}
where 
\begin{align}
	\label{e.calBo}
	\calB_{p,1}(t) := \frac1{\Gamma(1-\alpha)}\int_{1}^{\infty} s^{-\alpha}\Ex[U^n e^{-sU}]\,\d s,
	\quad
	U:=\calZ(2t,0)e^{\frac{t}{12}}.
\end{align}
For this term $ \calB_{p,1}(t) $ we bound
\begin{align*}
	0 \leq \calB_{p,1}(t)
	=
	\frac1{\Gamma(1-\alpha)}\int_{1}^{\infty} s^{-n-\alpha}\Ex[(sU)^n e^{-sU}]\,\d s
	\leq
	\frac1{\Gamma(1-\alpha)}  \, \sup_{x \geq 0} \big\{ x^{n}e^{-x} \big\} \, \frac1{n+\alpha-1}.
\end{align*}
Recognize $ n+\alpha-1=p $, and apply the bounds $ \frac1{\Gamma(1-\alpha)} \leq \Con $, for $ \alpha\in[0,1) $,
and $ \sup_{x \geq 0} \{ x^{n}e^{-x} \} \leq n^n $.
\begin{align}
	\label{e.calB0bd}
	|\calB_{p,1}(t)|
	\leq
	\Con \, p^{-1} \, n^n.
\end{align}
The bound~\eqref{e.calB0bd} does not grow with $ t $, and hence $ \calB_{p,1}(t) $ will be a subdominant term.

Next, we wish to take $ \partial_s^n $ in the Fredholm determinant expansion~\eqref{e.fred} and develop the corresponding series.
Assuming (justified later) the derivative can be passed into the sum, we have
\begin{align}
	\notag
	&\frac{(-1)^n}{\Gamma(1-\alpha)}\int_{0}^{1} s^{-\alpha} \,\partial_s^n \,\Ex[e^{-s\calZ(2t,0)e^{\frac{t}{12}}}]\,\d s
\\
	\label{e.derived}
	=&
	\frac{(-1)^n}{\Gamma(1-\alpha)} \int_{0}^{1} s^{-\alpha} \,\partial_s^n \Big( \sum_{L=1}^\infty (-1)^L \tr(K_{s,t}^{\wedge L}) \Big) \d s
	=
	\frac{(-1)^n}{\Gamma(1-\alpha)} \int_0^1 s^{-\alpha}\, \sum_{L=1}^\infty (-1)^L\, \partial_s^n \,\tr(K_{s,t}^{\wedge L}) \, \d s. 
\end{align}
The passing of derivatives into sums will be justified in Lemma~\ref{l.diff.term},
and in Sections~\ref{sec.trace} and \ref{sec.higher.justify}, we will show that $ \tr(K_{s,t}^{\wedge L}) $ is infinitely differentiable in $ s $. As it turns out, the $ L=1 $ term dominates. We then let
\begin{align}
	\label{e.calA}
	\til{\calA}_{p}(t) &:= \frac{(-1)^{n+1}}{\Gamma(1-\alpha)}\int_0^1 s^{-\alpha}\,\partial^n_s\, \tr(K_{s,t}) \, \d s,
\\
	\label{e.calBL}
	\calB_{p,L}(t) &:= \frac{(-1)^{n+L}}{\Gamma(1-\alpha)}\int_0^1 s^{-\alpha}\,\partial^n_s\, \tr(K^{\wedge L}_{s,t}) \, \d s, \quad L\geq 2
\end{align}
denote the leading and higher order terms.

In the following we will work with \revi{the} Schatten \revi{norms} of operators. Recall that, for $ u\in[1,\infty] $ and for a compact \revi{operator} $ T $ on $ L^2(\R_{\geq 0}) $,
the \textbf{$ u $-th Schatten norm} of $ T $ is defined as 
\begin{align*}
	\norm{T}_u := \big( \tr(T^*T)^{u/2} \big)^{1/u} = \Big( \sum_{i=1}^\infty s_i(T)^{u} \Big)^{1/u},
\end{align*}
with the convention $ \norm{T}_\infty:= \lim_{u\to\infty} \norm{T}_u $,
where $ s_i(T) $, $ i\in\Z_{>0} $, are the singular values of $ T $.
In particular, $ u=1 $ gives the trace norm, $ u=2 $ gives the Hilbert--Schmidt norm,
and $ u=\infty $ gives the operator norm
$ \norm{T}_\op := \sup\{ \frac{\normm{Tf}}{\normm{f}} : f\in L^2(\R_{\geq 0})\setminus\{0\} \} $, where $ |f| := (\int_{0}^\infty |f(x)|^2 \, \d x)^{1/2} $ denotes the norm on $ L^2(\R_{\geq 0})$.
The Schatten norm decreases in $ u $, so the trace norm is the strongest among all $ u\in[1,\infty] $.
We will use the following `square-root trick' to evaluate the trace norm of some operators.

\begin{lemma}\label{l.sqtrick}
Consider a square-integrable kernel $\revi{J(r,y)}$ \revi{with} $ \revi{\int_{\R_+} ( \int_{\R} |J(r,y)|^2 \, \d r) \d y  <\infty }$.
Then the integral operator $ T $ on $ L^2(\R_{\geq 0}) $ with the kernel
\begin{align*}
	T(x,y) := \int_{\R} \revi{\bar{J}(r,x)} \, J(r,y) \, \d r
\end{align*}
is positive and trace-class, with $ \revi{\tr(T) = \norm{T}_1 = \int_{\R_+} (\int_{\R}|J(r,y)|^2  \, \d r) \d y  }  $.
\end{lemma}
\begin{proof}
It is more convenient to embed $ T $ into operators on $ L^2(\R) $.
We do this by setting the kernel
\begin{align*}
	T(x,y) := \ind_{\R_{\geq 0}}(x) \ind_{\R_{\geq 0}}(y) \, \int_{\R} \revi{\bar{J}(r,x)} \, J(r,y) \, \d r
\end{align*}
to be zero outside $ (x,y)\in\R_{\geq 0}^2 $.
This way we have the factorization $ T=J^*J $, where $ J $ is an operator on $ L^2(\R) $
with kernel $ \ind_{\R_{\geq 0}}(y) J(r,y) $.
The square integrability of $ J(r,y) $ guarantees that the operator $ J $ is Hilbert--Schmidt,
and the Cauchy--Schwartz inequality $ \norm{T_1T_2}_1 \leq \norm{T_1}_2\norm{T_2}_2 $ applied with $ \revi{T_1=J^*, T_2=J} $ concludes that $ T $ is trace-class, \revi{whence $\tr(T)= \int_{\R_+}(\int_{\R} |J(r,y)|^2 \, \d r) \d y $ by Theorem 3.1 in \cite{brislawn1991traceable}. The factorization $ T=J^*J $ implies that $ T $ is positive, whence  $ \tr(T) = \norm{T}_1$}.
\end{proof}
\noindent{}Lemma~\ref{l.sqtrick} applied with $\revi{ J(r,y) = \ai(y+r)(1+\frac1se^{-t^{1/3}r})^{-1/2} }$ proves that the operator $ K_{s,t} $ (defined in~\eqref{e.K}) is positive and trace-class.

Much of our subsequent analysis boils down to estimating integrals involving the Airy function $ \ai(x) $. 
Here we prepare two technical lemmas that will be frequently used. To setup the notation, set
\begin{align}
	\label{e.aisq}
	\aisq(y) := \int_y^\infty \ai^2(x) \, \d x.
\end{align}
Using the Airy differential equation, one can explicitly compute the antiderivative of $\ai(x)^2$ to get $\aisq (y)=\ai'(y)^2-y\ai(y)^2$. Using known expansions of $ \ai(x), \ai'(x) $ for $ |x| \gg 1 $, e.g., Equation (1.07), (1.08), and (1.09) in Chapter 11 of \cite{olver1997asymptotics}, we have that, for all $ y \geq 0 $ and for some universal $ \Con>0 $,
\begin{align} 
	\label{u-low}
	\frac1{\Con}(\sqrt{|y|}+1) \le &\aisq ( -y) \le \Con\,(\sqrt{|y|}+1),
\\
	\label{u-up}
	\tfrac1{\Con\, (y+1)}e^{-\frac43y^{3/2}} \le &\aisq ( y) \le \tfrac\Con{y+1}e^{-\tfrac43y^{3/2}}.
\end{align}
Also consider
\begin{align}
	\label{e.expq}
	\expf_q(x):= qx^2-\tfrac{4}{3}x^{3},
\end{align}
which enjoys the property
\begin{align}
	\label{e.expq.property}
	\expf_q(x) \text{ increases on } x\in[0,\tfrac{q}{2}] \text{ and decreases on } x\in[\tfrac{q}{2},\infty), 
	\quad
	\revi{\expf_q(\tfrac{q}{2}) = \tfrac{q^3}{12}}.
\end{align}

\begin{lemma}\label{ap-int1} 
Fix $ t_0,q_0\in (0,\infty)$. There exists a constant $\Con(t_0,q_0)>0$, such that for all $t\ge t_0 $ and $q\ge q_0$, 
\begin{align} \label{ap-int-eq}
	\frac1{\Con(t_0,q_0)} t^{-7/6}q^{-3/2}e^{\frac{q^3t}{12}} 
	\le 
	\int_{\R} e^{qrt}\aisq ( t^{2/3}r)\,\d r
	\le 
	{\Con(t_0,q_0)} t^{-7/6}q^{-3/2}e^{\frac{q^3t}{12}}. 
\end{align}
\end{lemma}
\begin{proof}
Let us first give a heuristic of the proof. The idea is to apply Laplace's method.
We seek to approximate $ \int_{\R} e^{qrt}\aisq ( t^{2/3}r)\,\d r $ by $ \int_{\R} e^{tg_q(r)}\d r $, for some appropriate function $ g_q(r) $,
and search the maximum of $ g_q(r) $ over $ r\in\R $.
The bounds of $\aisq $ from \eqref{u-low} and \eqref{u-up} suggest $\log \aisq ( t^{2/3}r) \approx -\frac43tr_{+}^{3/2}$ and $ g_q(r) = qr-\frac43r_+^{3/2}  $.
This function achieves a maximum of $ q^3/12 $ at $ r=q^2/4 $, which gives the exponential factor $ \exp(\frac{q^3t}{12}) $. 
The prefactor $ t^{-7/6}q^{-3/2} $ can then be obtained from
localizing the integral around $ r=q^2/4 $ and using~\eqref{u-up} to approximate the integral as a Gaussian integral.

We now start the proof. Fix $ t_0,q_0>0 $. To simplify notation, throughout this proof we write $\Con=\Con(t_0,q_0)>0$,
and for positive functions $ f_1(a,b,\ldots), f_2(a,b,\ldots) $, we write $ f_1 \sim f_2 $ if they bound each other by a \revi{constant multiple}, i.e.,
\begin{align*}
	\tfrac{1}{\Con} f_2(a,b,\ldots) \leq f_1(a,b,\ldots) \leq \Con f_2(a,b,\ldots),
\end{align*}
within the specified ranges of the variables $ a,b,\ldots $. Set $ \revi{\rho} :=\frac{q_0}4 $. 
Divide $ \int_{\R} e^{qrt}\aisq ( t^{2/3}r)\,\d r $ into three regions and let $ \calI_1,\calI_2 $, and $\calI_3 $ denote the respective integrals:
\begin{align}\label{division}
	\hspace{-20pt}
	\Bigg( \int_{[(\frac{q}{2}-\revi{\rho})^2,(\frac{q}2+\revi{\rho})^2]}+\int_{\R_-} + \int_{\R_{\geq 0}\setminus[(\frac{q}{2}-\revi{\rho})^2,(\frac{q}2+\revi{\rho})^2]}  \Bigg) e^{qrt}\aisq ( t^{2/3}r)\d r
	:=
	\calI_1(q,t) + \calI_2(q,t)+ \calI_3(q,t).
\end{align}
As suggested by the preceding heuristics, we anticipate $ \calI_1(q,t) $ to dominate. We begin with estimating this term.
Recall $ \expf_q(x) $ from~\eqref{e.expq}. The bounds from \eqref{u-up} gives, for all $ r,t\in\R_{\geq 0} $,
\begin{align}
	\label{both-sides}
	e^{qrt}\aisq ( t^{2/3}r) \sim \frac{e^{t\expf_q(\sqrt{r})}}{1+t^{2/3}r}.
\end{align}
The function $\revi{U_q(x)}$ attains a maximum of $\frac{q^3}{12}$ at $x=\frac{q}{2}$ and 
$
	\expf_q(x)-\frac{q^3}{12}=-(x-\frac{q}2)^2(\frac43(x-\frac{q}2)+q).
$
Integrate both sides of~\eqref{both-sides} over $[(\frac{q}{2}-\revi{\rho})^2,(\frac{q}2+\revi{\rho})^2]$ and make a change of variable $\sqrt{r}-\frac{q}2\mapsto x $. We get, for all $ q,t\in\R_{\geq 0} $,
\begin{align*}
	\calI_1(q,t) \sim e^{\frac{q^3t}{12}}\int_{-\revi{\rho}}^{\revi{\rho}} \frac{2(x+\frac{q}{2})e^{-tx^2(\frac43x+q)}}{1+t^{2/3}(x+\frac{q}2)^2}\d x.
\end{align*}
The choice $\revi{\rho}=\frac{q_0}{4}$ guarantees that for all $x\in [-\revi{\rho},\revi{\rho}]$ and for all $q\ge q_0$, 
we have $\frac{q}\Con \le  \frac43x+q, x+\frac{q}2 \le \Con q$. 
Then for all $t\ge t_0 $ and $q\ge q_0$, there exists $\Con>0$ such that for $x\in [-\revi{\rho},\revi{\rho}]$,
\begin{align}\label{j3}
	\frac{1}{\Con t^{2/3}q}e^{-\Con qt x^2} \le \frac{2(x+\frac{q}{2})e^{-tx^2(\frac43x+q)}}{1+t^{2/3}(x+\frac{q}2)^2} \le \frac{\Con}{t^{2/3}q}e^{-\frac1{\Con}qtx^2}.
\end{align}
Integrate \eqref{j3} over $[-\revi{\rho},\revi{\rho}]$
and use $ \int_{-\revi{\rho}}^{\revi{\rho}} e^{-qx^2t} \, \d x \sim (tq)^{-1/2} $, for all $ t\geq t_0 $ and $ q \geq q_0$.
We now obtain, for $ t\geq t_0 $ and $ q \geq q_0 $,
\begin{align}\label{fir}
	\calI_1(q,t) \sim t^{-7/6}q^{-3/2} e^{\frac{q^3t}{12}}.
\end{align}
Having settled the asymptotics of $ \calI_1(q,t) $, we now turn to $ \calI_2(q,t),\calI_3(q,t) $.
For $\calI_2(q,t)$, use \eqref{u-low} to get
\begin{align} \label{i1}
	0 \leq \calI_2(q,t)\le \Con\int_{-\infty}^{0} e^{qrt}(\sqrt{t^{2/3}|r|}+1)\d r \leq \Con q^{-3/2}t^{-7/6}+\Con q^{-1}t^{-1}.
\end{align}
As for $\calI_3(q,t)$, 
\revi{integrate} both sides of \eqref{both-sides} over $ \R_{\geq 0}\setminus[(\frac{q}{2}-\revi{\rho})^2,(\frac{q}2+\revi{\rho})^2] $ and then make the change of variable $\sqrt{r} \mapsto x$ to get
\begin{align} \label{I2}
	0 \leq \calI_3(q,t) \le \Con e^{\frac{q^3t}{12}} \int_{\R_{\geq 0}\setminus[\revi{(\frac{q}{2}-\revi{\rho}),(\frac{q}2+\revi{\rho})}]} \frac{2xe^{-\frac{t}3(x-\frac{q}2)^2(4x+q)}}{1+t^{2/3}x^2}\d x.
\end{align}
For $x\in \R_{\geq 0}\setminus[\revi{(\frac{q}{2}-\revi{\rho}),(\frac{q}2+\revi{\rho})}] $, we have $(x-\frac{q}2)^2 \ge \revi{\rho}^2$. \revi{The AM-GM inequality inequality gives $1+t^{2/3}x^2 \ge 2t^{1/3}x$, and equivalently $\frac{2x}{1+t^{2/3}x^2} \le t^{-1/3}$}. 
Applying these bounds on the r.h.s.\ of~\eqref{I2} and then releasing the region of integration to $\R_{\geq 0}$, we get that
\begin{align}\label{j2}
	\calI_3(q,t) \le \Con e^{\frac{q^3t}{12}} t^{-4/3}\revi{\rho}^{-2} e^{-\frac{qt\revi{\rho}^2}3}.
\end{align}
It is straightforward to check that the r.h.s. of~\eqref{i1} and \eqref{j2} 
can be further bounded by $ \Con\, t^{-7/6}q^{-3/2}e^{\frac{q^3t}{12}} $, for all $ t \geq t_0 $ and $ q \geq q_0 $.
Hence
\begin{align*}
	0 \leq \calI_2(q,t)+\calI_3(q,t)
	\le 
	\Con\, t^{-7/6}q^{-3/2}e^{\frac{q^3t}{12}}.
\end{align*} 
This together with~\eqref{fir} gives the desired result~\eqref{ap-int-eq}.
\end{proof}

\begin{lemma}\label{ap-int}
Recall $ \expf_q $ from~\eqref{e.expq}.
There exists a constant $\Con=\Con(t_0,q_0)>0$ such that for all $t\ge t_0 $, $q\ge q_0$, and $y\in [0,\infty] $, 
\begin{align} \label{ap-int2-eq}
	\int_{-\infty}^{y} e^{qrt}\aisq ( t^{2/3}r)\d r 
	\le 
	\Con(t_0,q_0) \, t^{-5/6} \, \exp\big( \, t \expf_q( \,\min\{\sqrt{y},\tfrac{q}{2}\}\,)\,\big). 
\end{align}
\end{lemma}
\begin{remark}
The prefactor $ t^{-5/6} $ in \eqref{ap-int-eq} is likely suboptimal, but suffices for our subsequent analysis.
\end{remark}
\begin{proof}
\revi{When $y\in [\frac{q^2}{4},\infty]$, we release the range of integration of the l.h.s.\ of \eqref{ap-int2-eq} to $\R$ and use the upper bound in Lemma~\ref{ap-int1}. Observe that $U_q\left(\min\{\sqrt{y},\tfrac{q}{2}\}\right)=\frac{q^3}{12}$ and $t$ and $q$ are bounded below by $t_0$ and $q_0$. Absorb $t^{-1/3}$ and $q^{-3/2}$ in the constant $C(t_0,q_0)$ to get the desired bound in \eqref{ap-int2-eq}.}

\revi{Moving onto $ y\in[0,q^2/4) $},
from \eqref{i1} we already have a bound on $ \int_{-\infty}^{0} e^{qrt}\aisq ( t^{2/3}r)\d r $ of the desired form.
Hence, it suffices to bound for $ \int_{0}^y e^{qrt}\aisq ( t^{2/3}r)\d r $. 
From $\eqref{both-sides}$, make a change of variable $\sqrt{r} \mapsto x$, and in the result bound $\frac{2x}{1+t^{2/3}x^2}\le t^{-1/3}$.
We have
\begin{align}\label{exp}
	\int_{0}^{y} e^{qrt}\aisq ( t^{2/3}r)\d r
	\le 
	\Con \int_{0}^{y} \frac{e^{t\expf_q(\sqrt{r})}}{1+t^{2/3}r}\d r=\Con \int_{0}^{\sqrt{y}} \frac{2xe^{t\expf_q(x)}}{1+t^{2/3}x^2}\d x \le \Con t^{-1/3} \int_{0}^{\sqrt{y}} e^{t\expf_q(x)}\d x.
\end{align}
We next bound the last expression in \eqref{exp} in two cases.

\smallskip

\noindent\textbf{Case 1. $0\le y\le \frac{q^2}{16}$.}  
Since $\expf_q''(x)=2q-8x$ is positive for $x\in [0,\frac{q}{4}) $, the derivative $\expf_q'(x)=2x(q-2x)$ is increasing in $x\in [0,\frac{q}{4}]$. 
\revi{Hence, for any $z\in [0,\frac{q}{4}]$, $\expf_q'(z) \le \expf_q'(\tfrac{q}{4})=\frac{q^2}{4}$. Thus, for any $x\in [0,\sqrt{y}]$, we have $z_* \in [x,\sqrt{y}]$ for which
\begin{align}\label{roll}
	\expf_q(\sqrt{y})-\expf_q(x)=\expf_q'(z_*)(\sqrt{y}-x) \le \expf_q'(\tfrac{q}{4})(\sqrt{y}-x)=\tfrac{q^2}{4}(\sqrt{y}-x).
\end{align}}
Using~\eqref{roll} to bound $ \exp(t \expf_q(x)) $ and integrating the result over $ x\in[0,\sqrt{y}] $ gives
\begin{align}
\label{e.case1}
\int_{0}^{\sqrt{y}} e^{t\expf_q(x)}\d x \le \int_0^{\sqrt{y}}e^{t\expf_q(\sqrt{y})-\frac{1}4q^2t(\sqrt{y}-x)}\d x \le \int_{-\infty}^{\sqrt{y}}e^{t\expf_q(\sqrt{y})-\frac{1}4q^2t(\sqrt{y}-x)}\d x 
\le \frac4{q^2t}e^{t\expf_q(\sqrt{y})}.
\end{align}

\smallskip

\noindent\textbf{Case 2. $\frac{q^2}{16} \le y\le \frac{q^2}{4}$.} In this case we have $q\ge 2\sqrt{y}$, which gives
\begin{align*}
	\expf_q(\sqrt{y})-\expf_q(x)
	=
	q(y-x^2)-\tfrac43(y^{3/2}-x^{3}) 
	\ge 2\sqrt{y}(y-x^{2})-\tfrac43(y^{3/2}-x^{3})
	=
	\tfrac23(\sqrt{y}-x)^2(\sqrt{y}+2x).
\end{align*}
In the last expression, further use $ \sqrt{y}+2x \ge \sqrt{y}\ge \frac{q}{4} $ to get
\begin{align}\label{roll.}
	\expf_q(\sqrt{y})-\expf_q(x) \ge \tfrac{q}6(\sqrt{y}-x)^2.
\end{align}
Using~\eqref{roll.} to bound $ \exp(t \expf_q(x)) $ and integrate the result over $ x\in[0,\sqrt{y}] $ gives
\begin{align}
\label{e.case2}
\int_{0}^{\sqrt{y}} e^{t\expf_q(x)}\d x \le \int_0^{\sqrt{y}}e^{t\expf_q(\sqrt{y})-\frac{1}6qt(\sqrt{y}-x)^2}\d x \le \int_{-\infty}^{\sqrt{y}}e^{t\expf_q(\sqrt{y})-\frac{1}6qt(\sqrt{y}-x)^2}\d x 
\le \sqrt{\frac{\Con}{qt}}e^{t\expf_q(\sqrt{y})}.
\end{align}

Combining \eqref{e.case1} and \eqref{e.case2} and inserting the bounds into \eqref{exp} gives the desired result.
\end{proof}

\section{Estimates for the leading term} \label{sec.trace}
The goal of this section is to obtain the $ t\to\infty $ asymptotics of $ \til\calA_p(t) $ defined in~\eqref{e.calA}, 
accurate up to constant multiples.

Let us first settle the differentiability in $ s $ of the operator $ K_{s,t} $, defined in~\eqref{e.K}. Recall $K_{s,t}(x,y)$ from \eqref{e.K}, then perform a change of variable $r\mapsto t^{2/3}r$ to get
\begin{align}
\label{e.K.}
K_{s,t}(x,y) & =t^{2/3}\int_{\R} \ai(x+t^{2/3}r)\ai(y+t^{2/3}r)v(s,t,r)\d r,
\\
\label{e.v}
v(s,t,r) &:= \frac{1}{1+\frac{1}{s}e^{-rt}}.
\end{align} 
Formally differentiating the kernel $ K_{s,t}(x,y) $ in~\eqref{e.K} in $ s $ suggests that the $ n $-th derivative of $ K_{s,t} $ should have kernel
\begin{align}
	\label{e.Kn}
	K^{(n)}_{s,t}(x,y) 
	&:=
	t^{2/3}\int_{\R} \ai(x+t^{2/3}r) \ai(y+t^{2/3}r) \partial^n_s v(s,t,r) \, \d r,
\end{align}
with the convention $ K^{(0)}_{s,t}(x,y) := K_{s,t}(x,y) $. Differentiating \eqref{e.v} with respect to $s$ we get 
\begin{align}
\label{e.partial.v}
\partial^n_s v(s,t,r)=\frac{(-1)^{n-1}n!e^{-rt}}{(s+e^{-rt})^{n+1}}.
\end{align}
Since $ (-1)^{n-1}\partial^n_s v(s,t,r)>0 $, Lemma~\ref{l.sqtrick} applied with $\revi{ J(r,y)= \ai(y+t^{2/3}r)((-1)^{n-1}\partial^n_s v(s,t,r))^{1/2} }$ gives that $ (-1)^{n-1} K^{(n)}_{s,t} $ defines a positive trace-class operator on $ L^2(\R_{\geq 0}) $.

\begin{lemma} \label{l.K.diff}
For any $ n \in \Z_{\geq 0} $, $ u\in[1,\infty] $ and $ t>0 $, the operator $ K^{(n)}_{s,t} $ is differentiable in $ s $ at each $ s>0 $ in the $ u $-th Schatten norm,
with derivative being equal to $ K^{(n+1)}_{\revi{s,t}} $, i.e.,
\begin{align*}
	\lim_{s'\to s}\Big\Vert \frac{K^{(n)}_{s',t}-K^{(n)}_{s,t}}{s'-s} - K^{(n+1)}_{\revi{s,t}} \Big\Vert_u = 0.
\end{align*}
\end{lemma}
\begin{proof}
Since the Schatten norms decreases in $ u $, without lost of generality we assume $ u=1 $.
Fix $ n \in \Z_{\geq 0} $ and $ t>0 $, and set
$ 
	D_{s,s'} := \frac{1}{s'-s}(K^{(n)}_{s',t}-K^{(n)}_{s,t}) - K^{(n+1)}_{\revi{s,t}}. 
$
Use~\eqref{e.Kn} to express the kernel of $ D_{s,s'} $ as an integral involving $ \partial^n_{s} v $ and $ \partial^{n+1}_{s} v $,
and Taylor expand $ \partial^{n}_{\sigma} v(\sigma,t,r) $ around $ \sigma=s $ up to the first order, i.e.,
$
	\partial^{n}_{\sigma} v(s',r) - \partial^{n}_{\sigma} v(s,t,r) - (s'-s) \partial^{n+1}_{\sigma} v(s,t,r)
	=
	\frac12\int_{s}^{s'} (s'-\sigma) \partial^{n+2}_{\sigma} v(\sigma,t,r) \,\d\sigma.
$
We then get
\begin{align*}
	D_{s,s'}(x,y)
	&=
	t^{2/3}\int_{\R } \ai(x+t^{2/3}r) \ai(y+t^{2/3}r) \Big(  \frac{1}{2(s'-s)}\int_{s}^{s'} (s'-\sigma)\partial^{n+2}_\sigma v(\sigma,t,r) \,\d \sigma \Big) \d r
\\
	&=
	t^{2/3}\int_{\R_{\geq 0}\times\R } \ai(x+t^{2/3}r) \ai(y+t^{2/3}r) \sgn(s'-s) \ind_{|(s,s')|}(\sigma)  \frac{(s'-\sigma)}{2(s'-s)}\partial^{n+2}_\sigma v(\sigma,t,r) \, \d \sigma \d r,
\end{align*}
where $ |(s,s')| := (s,s') $ for $ s<s' $ and $ |(s,s')| := (s',s) $ for $ s'<s $.

Our goal is to show that $ \norm{D_{s,s'}}_1 $ converges to zero as $ s'\to s $.
As seen from \eqref{e.partial.v}, we have $ (-1)^{n+1}\partial^{n+2}_\s v(\s,t,r) >0 $.
Applying Lemma~\ref{l.sqtrick} with $ \revi{ J(r,y)=\ai(y+t^{2/3}r)((-1)^{n+1}\partial^{n+2}_\s v(\s,t,r))^{-1/2} }$ gives
\begin{align}
	\label{e.Ks.diff.1}
	\Vert D_{s,s'}\Vert_1 
	= 
	t^{2/3}\int_{\R^2_{\ge 0}\times\R } \ai^2(x+t^{2/3}r) \ind_{|(s,s')|}(\s) \Big| \frac{(s'-\s)}{2(s'-s)}\partial^{n+2}_\sigma v(\sigma,t,r) \Big| \d \s \d x \d r,
\end{align}
provided that the last integral converges.
To check the convergence, recognizing $\int_{\R_{\ge 0}} \ai(x+t^{2/3}r)^2\d x=\aisq(t^{2/3}r)$ substituting \eqref{e.partial.v} into \eqref{e.Ks.diff.1}, bound $(\revi{\s+e^{-rt}})^{n+3} \ge e^{-(n+3)rt}$, and $ |\frac{s'-\sigma}{2(s'-s)}| \leq \frac12 $, 
\begin{align}
	\label{e.Ks.diff.2}
	(\text{r.h.s.\ of } \eqref{e.Ks.diff.1})
	\leq
	\frac{1}{2}\revi{(n+2)!}t^{2/3} \int_{\R_{\ge 0}\times\R} e^{\revi{(n+2)rt}}\aisq(t^{2/3}r) \ind_{|(s,s')|}(\s) | \d \s \d r.
\end{align}
By Lemma \ref{ap-int} with $y\mapsto \infty$, the r.h.s.\ of~\eqref{e.Ks.diff.2} is finite for each $ s'\in\R_{\geq 0} $.
From this and the dominated convergence theorem, we conclude the desired result
$
	\Vert D_{s,s'} \Vert_1
	\leq
	(\text{r.h.s.\ of } \eqref{e.Ks.diff.2})
	\to
	0,
$
as $ s'\to s $.
\end{proof}

Applying Lemma~\ref{l.K.diff} with $ u=1 $ gives $ \partial^n_s \tr(K_{s,t})=\tr(K_{s,t}^{(n)}) $.
Further, since the operator $ K_{s,t}^{(n)} $ has a continuous kernel given in~\eqref{e.Kn} \revi{and is a trace-class operator}, the trace can be written as $ \tr(K_{s,t}^{(n)}) = \int_0^\infty K_{s,t}^{(n)}(x,x)\, \d x $ (\revi{see Corollary 3.2 in \cite{brislawn1991traceable}}).
To evaluate the last integral, insert \eqref{e.v} into \eqref{e.K.} and \eqref{e.partial.v}  into \eqref{e.Kn} to get
\begin{align}
	\label{e.K.int}
	\tr(K_{s,t})
	&=
	t^{2/3}\int_{\R} \frac{1}{1+\frac{1}{s}e^{-rt}} \aisq(t^{2/3}r) \, \d r,
\\
	\label{e.Kn.int}
	\partial^n_s \tr(K_{s,t})
	&=
	\tr(K^{(n)}_{s,t})
	=
	t^{2/3} \int_{\R} \frac{(-1)^{n-1} \, n! \,e^{-rt}}{(s+e^{-rt})^{n+1}} \aisq(t^{2/3}r) \, \d r,
	\quad
	n\in\Z_{> 0},
\end{align} 
where $ \aisq(y) $ is defined in~\eqref{e.aisq}.
Armed with the expressions \eqref{e.K.int} and \eqref{e.Kn.int}, we now proceed to establish the desired asymptotics of $ \til\calA_{p}(t) $.
Recall from~\eqref{e.calA} $ \til\calA_{p}(t) $ involves an integral over $ s\in[0,1] $.
It is convention to write it as the difference of an integral over $ s\in[0,\infty) $ and over $ s\in[0,1] $:
\begin{align}
	\label{e.calA.decomp}
	\til{\calA}_{p}(t) &= \calA_{p}(t) - \hat{\calA}_{p}(t),
	\\
	\label{e.calA.}
	\calA_{p}(t)
	&:= \frac{(-1)^{n+1}}{\Gamma(1-\alpha)}\int_0^\infty s^{-\alpha}\,\partial^n_s\, \tr(K_{s,t}) \, \d s,
	\quad
	\hat\calA_{p}(t)
	:= \frac{(-1)^{n+1}}{\Gamma(1-\alpha)}\int_1^\infty s^{-\alpha}\,\partial^n_s\, \tr(K_{s,t}) \, \d s,
\end{align}
where $ n:=\lfloor p\rfloor+1\in\Z_{> 0} $ and $ \alpha:=p-\lfloor p\rfloor\in [0,1) $.
\begin{proposition} \label{trace-estimate} 
Fix any $ t_0,p_0>0 $. 
There exists $ \Con=\Con(t_0,p_0)>0 $ such that for all $t\ge t_0$ and $p\ge p_0$,
\begin{align}
	\label{e.calA.bd}
	\tfrac1{\Con} p^{-\frac32}\Gamma(p+1)\,t^{-\frac12}\, e^{\frac{p^3t}{12}}
	\leq
	&\calA_p(t) \leq \Con p^{-\frac32}\Gamma(p+1)\,t^{-\frac12}\, e^{\frac{p^3t}{12}},
\\
	\label{e.calAh.bd}
	&|\hat{\calA}_p(t)| \leq \Gamma(p+1)\Con.
\end{align}
\end{proposition}
\begin{proof} 
Fix $ t_0,p_0>0 $. To simplify notation, throughout this proof we assume $ t \geq t_0 $ and $ p \geq p_0 $ and write $ \Con=\Con(t_0,p_0) $.
Referring to~\eqref{e.calA} and \eqref{e.Kn.int}, we set
\begin{align}
	\label{e.trace.int}
	\phi_{p,t}(s) := \frac{n!t^{2/3}}{\Gamma(1-\alpha)} s^{-\alpha} \int_{\R} \frac{e^{-rt}\aisq ( t^{2/3}r)}{(s+e^{-rt})^{n+1}} \, \d r
\end{align}
so that $ \calA_p(t) = \revi{\int_0^\infty} \phi_{p,t}(s) \d s $ and $ \hat\calA_p(t) = \int_1^\infty \phi_{p,t}(s) \d s $.

To estimate~$ \calA_p(t) = \revi{\int_0^\infty} \phi_{p,t}(s) \d s $, integrate~\eqref{e.trace.int} over $ s\in[0,\infty) $ to get
\begin{align*}
	\int_{0}^{\infty} \phi_{p,t}(s) \, \d s 
	= 
	\frac{n!t^{2/3}}{\Gamma(1-\alpha)} \int_{\R} e^{-rt}\aisq ( t^{2/3}r) \Big( \int_{0}^{\infty} \frac{s^{-\alpha}\d s}{(s+e^{-rt})^{n+1}} \Big) \d r. 
\end{align*}
The inner integral on the right hand side can be identified with the Beta integral. 
Namely the change of variable $ v=\frac{s}{s+e^{-rt}}$ yields
\begin{align}\label{beta}
	\int_{0}^{\infty} \frac{s^{-\alpha}\d s}{(s+e^{-rt})^{n+1}} = e^{nrt+\alpha rt}\int_{0}^{1} v^{-\alpha}(1-v)^{n-1+\alpha}\d v = e^{nrt+\alpha rt}\frac{\Gamma(1-\alpha)\Gamma(n+\alpha)}{n!}.
\end{align}
This then gives
$
	\int_{0}^{\infty} \phi_{p,t}(s) \, \d s 
	= 
	t^{2/3}\Gamma(p+1)\int_{\R} e^{prt}\aisq ( t^{2/3}r)\d r.
$
The asymptotics of last integral is given by Lemma~\ref{ap-int1} with $q\mapsto p$.
From this we conclude the desired estimate~\eqref{e.calA.bd} of $ \calA_p(t) $.

Next we turn to $ \hat\calA_p(t) = \int_1^\infty \phi_{p,t}(s) \d s $.
Integrate~\eqref{e.trace.int} over $ s\in(1,\infty) $, divide the integral over $ r\in(-\infty,0] $ and $ r\in [0,\infty) $,
and for the former release the integral over $ s $ from $ s\in(1,\infty) $ to $ s\in[0,\infty) $.
This gives
$	
	0 \leq \int_1^\infty \phi_{p,t}(s) \, \d s \leq A_1 + A_2,
$
where
\begin{align*}
	A_1 :=
	\frac{n!t^{2/3}}{\Gamma(1-\alpha)} \int_{[0,\infty)\times(-\infty,0]} s^{-\alpha} \frac{e^{-rt}\aisq ( t^{2/3}r)}{(s+e^{-rt})^{n+1}} \, \d s \d r,
	\quad
	A_2 :=
	\frac{n!t^{2/3}}{\Gamma(1-\alpha)} \int_{(1,\infty)\times[0,\infty)} s^{-\alpha}\frac{e^{-rt}\aisq ( t^{2/3}r)}{(s+e^{-rt})^{n+1}} \, \d s \d r.
\end{align*}
For $ A_1 $ use \eqref{beta} and then the bound from Lemma~\ref{ap-int} with $q\mapsto p$ and $ y\mapsto 0 $.
We have
\begin{align}
\label{e.A1.bd}
	A_1
	= 
	t^{2/3}\Gamma(p+1) \int_{-\infty}^{0} e^{prt}\aisq ( t^{2/3}r)\, \d r \leq t^{-1/6}\Gamma(p+1) \Con.
\end{align}
For $ A_2 $, use $ s \geq 1 $ to bound $ s^{-\alpha} \frac{1}{(s+e^{-rt})^{n+1}} \leq s^{-n-1} $
and use the fact that $ \aisq  $ is decreasing (see~\eqref{e.aisq}) to bound $ \aisq ( t^{2/3}r) \leq \aisq ( 0) = \Con $.
Together with $ \frac{1}{\Gamma(1-\alpha)} \leq 1 $, for $ \alpha \in [0,1) $, we have
\begin{align}
	\label{e.A2.bd}
	A_2
	\leq
	\Con
	\frac{n!t^{2/3}}{\Gamma(1-\alpha)} \int_{1}^\infty s^{-1-n} \, \d s \int_0^\infty e^{-rt} \d r
	\leq
	(n-1)!t^{-1/3} \Con
	\leq
	t^{-1/3}\Gamma(p+1) \Con.	
\end{align}
The last inequality follows from the fact that $ \Gamma(y) $ is increasing for $ y\geq 1 $ to bound $ (n-1)!=\Gamma(n) \leq \Gamma(p+1) $. 
Using $t\geq t_0$ to bound $t^{-1/6},t^{-1/3} \le \Con$, the bounds~\eqref{e.A1.bd} and \eqref{e.A2.bd} together gives the desired bound for \eqref{e.calAh.bd}.
\end{proof}

\section{Bounds for higher order terms} \label{sec.higher}
To goal of this section is to establish bounds on the term $ \calB_{p,L}(t) $ defined in~\eqref{e.calBL}.
Along the way we will also justify passing derivatives into sums in~\eqref{e.derived}.

Recall from~\eqref{e.Kn} and Lemma~\ref{l.K.diff} that $ K^{(n)}_{s,t} $ is the $ n $-th derivative in $ s $ of $ K_{s,t} $.
To prepare for subsequent analysis, we provide bounds on $ \tr(K_{s,t}) $ and $ \tr(K^{(n)}_{s,t}) $. 

\begin{lemma} \label{l.trace} 
Recall $ \expf_q $ from~\eqref{e.expq}.
For any $t_0>0$,  there exists a constant $ \Con(t_0)>0 $ such that for all $\s\in[0,\infty] $, $t>t_0$, and $ n\in\Z_{>0} $,
\begin{align}
	\label{e.tr.bd}
	\big| \tr(K_{e^{-t\s},t}) \big| &\le \Con(t_0)\exp(t\expf_1(\min\{\sqrt{\s},\tfrac12\})-t\s),
\\
	\label{e.trn.bd}
	\big|\tr(K_{e^{-t\s},t}^{(n)})\big| &\le  n! \, \Con(t_0) \,\exp(t\expf_n(\min\{\sqrt{\s},\tfrac{n}{2}\})).
\end{align}
\end{lemma}
\begin{proof}
The starting point of the proof is the explicit expressions \eqref{e.K.int} and \eqref{e.Kn.int} of the traces.
In~\eqref{e.K.int}, set $ s=e^{-\sigma t} $ and divide the integral into $ r<\sigma $ and $ r>\sigma $ to get
\begin{align}\label{airy-parts}
	\tr(K_{e^{-t\s},\revi{t}})
	=
	t^{2/3} \Big( \int_{-\infty}^{\s} + \int_{\s}^{\infty} \Big)\frac{\aisq ( t^{2/3}r)\d r}{1+e^{t\s-tr}}
	:=
	\calI_1 + \calI_2.
\end{align}
For $ \calI_1 $ use $1+e^{t\s-tr} \ge e^{t\s-tr} $ and Lemma~\ref{ap-int} with $q=1$ and $y=\s$. We have, for all $ t \geq t_0 $,
\begin{align} \label{upb-int1}
	\calI_1 \le \Con(t_0) \, \exp(t\revi{U_1}(\min\{\sqrt{\s},\tfrac12\})-t\s).
\end{align}
The second integral $ \calI_2 $ can be calculated explicitly by using Airy differential equation, whereby
\begin{align} 
	\calI_2 =\int_{t^{2/3}\s}^\infty \aisq ( r)\d r:=g(t^{2/3}\s),
	\quad
	g(y)=\tfrac13(2y^2\ai(y)^2-2y\ai'(y)^2-\ai(y)\ai'(y)).
\end{align}
Using the known $ |y|\gg 1 $ asymptotics of $ \ai(y) $ and $ \ai'(y) $ (see Equations (1.07), (1.08), and (1.09) in Chapter 11 of \cite{olver1997asymptotics} for example), 
we obtain $ g(y) \le \Con \exp(-\frac43y^{3/2}) $ for all $y\ge 0$. 
Using~\eqref{e.expq.property} we further bound the exponent $ -\tfrac43y^{3/2} \le \expf_1(\min\{\sqrt{y},\tfrac12\})-y$ for all $y\ge0$.
From this we conclude~\eqref{e.tr.bd}.

Moving on, similarly to the preceding,
in~\eqref{e.K.int} we set $ s=e^{-\sigma t} $ and divide the integral into $ r<\sigma $ and $ r>\sigma $ to get
\begin{align*}
	|\tr(K_{e^{-t\s},\revi{t}}^{(n)})| 
	= n!t^{2/3} \Big( \int_{-\infty}^{\s}+\int_{\s}^{\infty} \Big) \frac{e^{-rt}\aisq ( t^{2/3}r)\d r}{(e^{-\s t}+e^{-rt})^{n+1}}
	:= \calJ_1 + \calJ_2.
\end{align*}
For $ \calJ_1 $, use $e^{-\s t}+e^{-rt} \ge e^{-rt}$ and Lemma~\ref{ap-int} with $q=n$ to get, for $ t\ge t_0 $, 
\begin{align*}
	\calJ_1 \le \int_{-\infty}^{\s} e^{nrt}\aisq ( t^{2/3}r)\d r  
	\le 
	n! \, \Con(t_0) \, t^{-5/6}\exp(t\expf_n(\min\{\sqrt{\s},\tfrac{n}{2}\})) 
	\le
	n! \, \Con(t_0) \, \exp(t\expf_n(\min\{\sqrt{\s},\tfrac{n}{2}\})).
\end{align*}
This gives the desired bound for showing~\eqref{e.trn.bd}.
As for $ \calJ_2 $, use $e^{-\s t}+e^{-rt} \ge e^{-\s t}$ and the fact that $ \aisq $ is non-increasing to get
\begin{align*} 
	\calJ_2 & \le e^{t(n+1)\s}\aisq ( t^{2/3}\s)\int_{\s}^{\infty} e^{-rt}\d r = t^{-1}e^{tn\s}\aisq ( t^{2/3}\s). 
\end{align*}
Further bounding $ \aisq (y) \le \Con \exp(-\tfrac43y^{3/2}) $ (by~\eqref{u-up}) gives $ \calJ_2 \leq t^{-1}_0 \exp(t\expf_n(\sqrt{\s})) $, for all $ t \geq t_0 $.
From~\eqref{e.expq.property} we have $ \expf_q(\sqrt{s}) \le \expf_q(\min\{\sqrt{s},q/2\}) $, for all $ \s,q>0 $.
From this we conclude $ \calJ_2 \leq t^{-1}_0 \exp(t\expf_n(\min\{\sqrt{s},\tfrac{n}2\})) $, for all $ t \geq t_0 $.
This completes the proof of \eqref{e.trn.bd}.
\end{proof}

\subsection{Interchange of sum and derivatives}\label{sec.higher.justify}
In this subsection, we show that the series 
\begin{align}
	\label{e.series}
	\sum_{L=1}^\infty (-1)^L \tr(K_{s,t}^{\wedge L})  
\end{align}
is infinitely differentiable in $ s $ and the derivative can be obtained by taking term-by-term differentiation.
To this end we will use the following standard criterion:
\begin{proposition}\label{p.termwise}
Let $ f_k(s) $, $ k\in\Z_{>0} $, be \revi{$(n+1)$ times} continuously differentiable functions on $ s\in[0,1] $, where $ n\in\Z_{>0} $.
If the series $ f(s) := \sum_{k=1}^\infty f_k(s) $ converges absolutely at each $ s\in[0,1] $,
and if the absolute derivative series $ \sum_{k=1}^\infty |\frac{\d^j~}{\d s^{j}}f_k(s)| $ converges uniformly over bounded intervals in $ [0,1] $, for all $ j=1,\ldots,n+1$,
then $ f $ is $ n $-th differentiable for all $s\in[0,1]$ with $ \frac{\d^j~}{\d s^{j}} f(s)= \sum_{k=1}^\infty \frac{\d^j~}{\d s^{j}}f_k(s) $, for all $ j=1,\ldots,n$.
\end{proposition}
\noindent
The proof of this proposition is \revi{standard by} applying Dini's theorem to the sequence $ \sum_{\ell=1}^k\int_0^s |\frac{\d^{j+1}~}{\d s^{j+1}}f_\ell(s)| \d s $.

Let us consider first the $ s $ derivative of $ \tr(K_{s,t}^{\wedge L}) $.
Recall from~\eqref{e.Kn} and Lemma~\ref{l.diff.term} that $ K^{(n)}_{s,t} $ denotes the $ n $-th $ s $ derivative of $ K_{s,t} $.
Fix any orthonormal basis $\{e_i\}_{i\ge 1}$ for $L^2(\R_{\geq 0})$ and write
\begin{align}
	\label{e.tr^L}
	\tr(K_{s,t}^{\wedge L}) = \sum_{i_1<\ldots<i_L} \hspace{-5pt} \big\langle e_{i_1}\wedge\ldots\wedge e_{i_L}, K_{s,t}e_{i_1}\wedge\ldots\wedge K_{s,t}e_{i_L} \big\rangle
	=
	\sum_{i_1<\ldots<i_L} \hspace{-5pt} \det\big( \big\langle e_{i_k}, K_{s,t} e_{i_\ell} \big\rangle \big)_{k,\ell=1}^L.
\end{align}
Formally taking $ \partial_s^n $ in~\eqref{e.tr^L} and passing (without justification at the moment)
the derivatives into the sum and inner product suggest that the following should hold
\begin{align*}
	\partial_s^n \tr(K_{s,t}^{\wedge L})
	=
	\sum_{i_1<\ldots<i_L} \sum_{\vec{m}\in\Msp(L,n) } \binom{n}{\vec{m}}\det\big( \big\langle e_{i_k}, K^{(m_{\revi{\ell}})}_{s,t} e_{i_\ell} \big\rangle \big)_{k,\ell=1}^L,
\end{align*}
where
\begin{align}
	\label{e.Msp}
	\Msp(L,n) &:= \big\{ \vec{m}=(m_1,\ldots,m_L)\in (\Z_{\geq 0})^L : m_1+\cdots+m_L=n \big\},
\\
	\label{e.binom}
	\binom{n}{\vec{m}} &:= \frac{n!}{m_1!\cdots m_L!}.
\end{align}
We now proceed to justify this formal calculation. Doing so requires an inequality. Recall that $ \norm{\Cdot}_2 $ denotes the Hilbert--Schmidt norm.

\begin{lemma}\label{perm-ineq} 
Fix any $ k\in\Z_{>0} $ and any permutation $ \pi\in \mathbb{S}_k $.
Let $ T_1,T_2,\ldots,T_k $ be \revi{self-adjoint} Hilbert--Schmidt operators on a separable \revi{Hilbert space} $ H $, and let $\{e_i\}_{i\ge 1}$ be any orthonormal basis. 
Then
\begin{align} \label{perm-eqn}
	\sum_{i_1,\ldots,i_k\in\Z_{>0}} \prod_{\ell=1}^k \big|\langle e_{i_\ell},T_{\pi(\ell)}e_{i_{\pi(\ell)}}\rangle \big| 
	\le 
	\prod_{i=1}^k \norm{T_i}_2.
\end{align}
\end{lemma} 
\begin{proof}
It suffices to prove~\eqref{perm-eqn} for the case when $ \pi $ is a cycle of length $ k $.
For general $\pi\in\mathbb{S}_k $, decompose it into cycles of smaller lengths and apply the result within each cycle.
Further, since the r.h.s.\ of~\eqref{perm-eqn} is symmetric in $ T_1,\ldots,T_k $, we may assume without \revi{loss} of generality $ \pi=(12\ldots k) $.
Under this assumption the l.h.s.\ of \eqref{perm-eqn} becomes
\begin{align} \label{perm-eqn.}
	\sum_{i_1,\ldots,i_k\in\Z_{>0}} \prod_{\ell=1}^k \big|\langle e_{i_\ell},T_{\ell+1}e_{i_{\ell+1}}\rangle \big|, 
\end{align}
with the convention $ T_{k+1} := T_1 $ and $ e_{i_{k+1}} := e_{i_1} $.

Let $ |\Cdot|_H $ denote the norm of the Hilbert space $ H $.
Apply the Cauchy--Schwarz inequality in~\eqref{perm-eqn.} over the sum $ i_2\in\Z_{>0} $,
and within the result recognize $\revi{(\sum_{i_2} |\langle e_{i_1}, T_{2} e_{i_{2}}\rangle|^2)^{1/2} = |T_2e_{i_1}|_H} $
and $\revi{ (\sum_{i_2} |\langle e_{i_2}, T_{3} e_{i_{3}}\rangle|^2)^{1/2} = |T_3e_{i_3}|_H }$.
We have
\begin{align*}
	\text{l.h.s.\ of }\eqref{perm-eqn.} 
	\leq
	\sum_{i_1,i_3,\ldots,i_k} \big|T_2e_{i_1}\big|_H \, \big|T_3e_{i_3}\big|_H \, \prod_{\ell=4}^k \big|\langle e_{i_\ell},T_{\ell+1}e_{i_{\ell+1}}\rangle \big|.
\end{align*}
Next apply the Cauchy-Schwarz inequality over the sum $ i_3\in\Z_{>0} $.
Within the result recognize $ (\sum_{i_3} |T_3e_{i_3}|_H^2)^{1/2} = \norm{T_3}_2 $
and $ (\sum_{i_3} |\langle e_{i_3}, T_{4} e_{i_{4}}\rangle|^2)^{1/2} = |T_4e_{i_4}|_H $.
We have
\begin{align*}
	\text{l.h.s.\ of }\eqref{perm-eqn.} 
	\leq
	\sum_{i_1,i_4,\ldots,i_k} \big|T_2e_{i_1}\big|_H \, \Norm{T_3e_{i_3}}_2  \, \big|T_4e_{i_4}\big|_H \prod_{\ell=5}^L \big|\langle e_{i_\ell},T_{\ell+1}e_{i_{\ell+1}}\rangle \big|.
\end{align*}
Continue this procedure through $ i_j $, $ j=4,\ldots,k $.
Each application of the the Cauchy-Schwarz inequality turns the preexisting $ |T_je_{i_j}|_H $
into $ \norm{T_j}_2 $ and produces $ |T_{j+1}e_{i_{j+1}}|_{H} $.
Finally, after the $ j=k $ step, an application of the Cauchy--Schwarz inequality over $ i_1 $
turns $ |T_2e_{i_1}|_H $ and $ |T_1e_{i_1}|_H $ into $ \norm{T_2}_2  $ and $ \norm{T_1}_2 $.
\end{proof}

\begin{lemma}\label{l.diff.term}
Let $ \Msp(L,n) $ be in~\eqref{e.Msp}. Fix $ L\in\Z_{>0} $, and fix any orthonormal basis $\{e_i\}_{i\ge 1}$ for $L^2(\R_{\geq 0})$.
For any $ t>0 $, the function $ s\mapsto \tr(K_{s,t}^{\wedge L})$ is infinitely differentiable at each $ s\in[0,1] $,
with
\begin{align}
	\label{e.K^n.diff}
	\partial_s^n \tr(K_{s,t}^{\wedge L})
	=
	\sum_{i_1<\ldots<i_L} \sum_{\vec{m}\in\Msp(L,n)}\binom{n}{\vec{m}} \, \det\big( \big\langle e_{i_k}, K^{(\revi{m_\ell})}_{s,t} e_{i_\ell} \big\rangle \big)_{k,\ell=1}^L,
\end{align}
where the r.h.s.\ converges absolutely uniformly over $ [0,1]\ni s $.
\end{lemma}
\begin{proof}
First, by the product rule of calculus we have
\begin{align*}
	\partial^n_s \, \det\big( \big\langle e_{i_k}, K_{s,t} e_{i_\ell} \big\rangle \big)_{k,\ell=1}^L
	=
	\sum_{\vec{m}\in\Msp(L,n)}\binom{n}{\vec{m}} \, \det\big( \partial_s^{\revi{m_\ell}} \,\big\langle e_{i_k}, K_{s,t} e_{i_\ell} \big\rangle \big)_{k,\ell=1}^L.
\end{align*}
By Lemma~\ref{l.K.diff} for $ u=\infty $, the derivatives on the r.h.s.\ can be passed into the inner product to give
\begin{align}
	\label{e.tr^L.diff}
	\partial^n_s \, \det\big( \big\langle e_{i_k}, K_{s,t} e_{i_\ell} \big\rangle \big)_{k,\ell=1}^L
	=
	\sum_{\vec{m}\in\Msp(L,n)}\binom{n}{\vec{m}} \, \det\big( \big\langle e_{i_k}, K^{(m_\ell)}_{s,t} e_{i_\ell} \big\rangle \big)_{k,\ell=1}^L.
\end{align}

We wish to apply Proposition~\ref{p.termwise} with $ \{f_k\}_{k=1}^\infty $
being an enumeration of $ \{ \det( \big\langle e_{i_k}, K_{s,t} e_{i_\ell} \big\rangle )_{k,\ell=1}^L \}_{i_1<\ldots<i_L} $.
The series in \eqref{e.tr^L} converges absolutely for each $ s\in[0,1] $ (with $ t\in(0,\infty) $ fixed) because $ K^{\wedge L}_{s,t} $ is trace-class.
Given the identity \eqref{e.tr^L.diff} for the derivative series,
it suffices to prove that the r.h.s.\ of~\eqref{e.K^n.diff} converges absolutely and uniformly over $ [0,1]\ni s $.
To this end, apply Lemma~\ref{perm-ineq} with $ k=L $ and $ T_i=K^{(m_{i})}_{s,t} $ to get
\begin{align*}
	\sum_{i_1<\ldots<i_L} \sum_{\vec{m}\in\Msp(L,n)}\binom{n}{\vec{m}} \, \Big| \det\big( \big\langle e_{i_k}, K^{(\revi{m_\ell})}_{s,t} e_{i_\ell} \big\rangle \big)_{k,\ell=1}^L \Big|
	\leq
	L! \sum_{\vec{m}\in\Msp(L,n)}\binom{n}{\vec{m}} \prod_{\revi{\ell=1}}^L \Norm{ K^{(\revi{m_\ell})}_{s,t} }_2. 
\end{align*}
Recall that $ (-1)^{m-1}K^{(m)}_{s,t} $ is a positive trace-class operator,
whereby $ \norm{ K^{(m)}_{s,t} }_2 \leq \norm{ K^{(\revi{m})}_{s,t} }_1 = |\tr(K^{(m)}_{s,t})| $ and
\begin{align}
	\label{e.l.diff.term.}
	\sum_{i_1<\ldots<i_L} \sum_{\vec{m}\in\Msp(L,n)}\binom{n}{\vec{m}} \, \Big| \det\big( \big\langle e_{i_k}, K^{(\revi{m_\ell})}_{s,t} e_{i_\ell} \big\rangle \big)_{k,\ell=1}^L \Big|
	\leq
	L! \sum_{\vec{m}\in\Msp(L,n)}\binom{n}{\vec{m}} \prod_{\revi{\ell=1}}^L \big| \tr(K^{(\revi{m_\ell})}_{s,t}) \big|. 
\end{align}
The bounds from Lemma~\ref{l.trace} guarantee that the r.h.s.\ of~\eqref{e.l.diff.term.} converges uniformly over $ [0,1]\ni s $, for fixed $ t>0 $.
\end{proof}

We now consider the $ s $ derivative of the series~\eqref{e.series}.
\begin{proposition}\label{interchange}
Let $ \Msp(L,n) $ be in~\eqref{e.Msp}.
For $ \vec{m}\in\Msp(L,n) $, set $ \revi{\vec{m}}_{>0} := \{k:m_k>0\} \subset\{1,\ldots,L\} $ and let $ |\vec{m}_{>0}| $ denotes the cardinality.
For any $ t>0 $, the series~\eqref{e.series} is infinitely differentiable in $ s\in[0,1] $, with
\begin{align} 
	\label{e.justfied}
	\partial^n_s \Big( \sum_{L=1}^\infty (-1)^L\tr(K_{s,t}^{\wedge L}) \Big)
	&=
	\sum_{L=1}^\infty (-1)^L \, \partial^n_s \, \tr(K_{s,t}^{\wedge L}),
\\
	\label{eq-interchange}
	\big|\partial_s^n \tr(K_{s,t}^{\wedge L}) \big|
	&\leq
	\sum_{\vec{m}\in\Msp(L,n)}\binom{n}{\vec{m}} 
	\frac{(|\vec{m}_{>0}|)!}{(L-|\vec{m}_{>0}|)!}
	\prod_{k=1}^L \big|\tr(K^{(m_k)}_{s,t})\big|.	
\end{align}
\end{proposition}
\begin{proof}
We will appeal to Proposition~\ref{p.termwise}, with the choice $ f_L(s) = (-1)^L\tr(K_{s,t}^{\wedge L}) $.
Doing so requires bounds on the derivatives series, which we achieve by using Lemma~\ref{l.diff.term}. 
This lemma holds for \emph{any} orthonormal basis, and here, with $ K_{s,t} $ being compact and symmetric, we specialize to the eigenbasis of $ K_{s,t} $.
Let $ \{ v_{i}\}_{i\geq 1} $ be an orthonormal basis of $ K_{s,t} $, with eigenvalue $ \lambda_{i} $.
Indeed $ v_{i} $ and $ \lambda_i $ depend on $ s,t $, but we omit such dependence since in the subsequent analysis we will \emph{not} vary $ s,t $.
Expand the determinant in~\eqref{e.K^n.diff} into a sum of permutations, and specialize to $ e_{i}=v_{i} $:
\begin{align}
	\label{e.K^n.diff.}
	\partial_s^n \tr(K_{s,t}^{\wedge L})
	=
	\sum_{i_1<\ldots<i_L} \sum_{\vec{m}\in\Msp(L,n)}\binom{n}{\vec{m}} \, 
	\sum_{\pi\in\mathbb{S}_L} \sgn(\pi)
	\prod_{k=1}^L \big\langle v_{i_k}, K^{(m_{\pi(i_k)})}_{s,t} v_{i_{\pi(i_k)}} \big\rangle.
\end{align}
Recall the convention $ K^{(0)}_{s,t} := K_{s,t} $.
Because of the eigenrelation $ K_{s,t}v_{i}=\lambda_{i}v_{i} $, the product in \eqref{e.K^n.diff.} vanishes unless \revi{$ \pi(r)=r $ for all $r\in \{k:m_k=0\} $}.
Such permutations can be reduced to permutations on the set $ \vec{m}_{>0} \subset\{1,\ldots,L\} $,
and we let $ \mathbb{S}(\vec{m}_{>0}) $ denote the subgroup of all such reduced permutations.
The preceding discussion brings~\eqref{e.K^n.diff.} to
\begin{align*}
	\partial_s^n \tr(K_{s,t}^{\wedge L})
	=
	\sum_{i_1<\ldots<i_L} \sum_{\vec{m}\in\Msp(L,n)}\binom{n}{\vec{m}} \, 
	\prod_{k:m_k=0} \lambda_{i_k}
	\sum_{\pi\in\mathbb{S}(\vec{m}_{>0})} \sgn(\pi)
	\prod_{k\in \vec{m}_{>0}} \big\langle v_{i_k}, K^{(m_{\pi(i_k)})}_{s,t} v_{i_{\pi(i_k)}} \big\rangle.
\end{align*}
To bound this expression, take absolute value and pass it into the sum and products on the r.h.s.,
bound the ordered sum $ \sum_{i_1<\ldots<i_L} $ by the symmetrized sum $ \frac{1}{(L-|\vec{m}_{>0}|)!} \sum_{i_k:m_k=0} \sum_{i_{\ell}:\ell\in\vec{m}_{>0}} $,
and then use $ \sum_{i}|\lambda_i|=\sum_{i}\lambda_i=\tr(K_{s,t})=\tr(K^{(0)}_{s,t}) $.
We have
\begin{align*}
	\big|\partial_s^n \tr(K_{s,t}^{\wedge L}) \big|
	\leq
	\sum_{\vec{m}\in\Msp(L,n)}\binom{n}{\vec{m}} 
	\frac{1}{(L-|\vec{m}_{>0}|)!} 
	\tr(K^{(0)}_{s,t})^{L-|\vec{m}_{>0}|}
	\sum_{\pi\in\mathbb{S}(\vec{m}_{>0})}
	\sum_{i_{\ell}:\ell\in\vec{m}_{>0}}
	\prod_{\ell\in \vec{m}_{>0}} \big| \big\langle v_{i_\ell}, K^{(m_{\pi(i_\ell)})}_{s,t} v_{i_{\pi(i_\ell)}} \big\rangle \big|.
\end{align*}
Now apply Lemma~\ref{perm-ineq} with $ k\mapsto|\vec{m}_{>0}| $ and with the $ T_i $'s being the $ K^{(m_k)}_{s,t} $'s,
and use $ \norm{K^{(m)}_{s,t}}_2 \leq \norm{K^{(m)}_{s,t}}_1 = |\tr(K^{(m)}_{s,t})| $.
We further obtain
\begin{align}
	\label{e.K^n.diff..}
	\big|\partial_s^n \tr(K_{s,t}^{\wedge L}) \big|
	\leq
	\sum_{\vec{m}\in\Msp(L,n)}\binom{n}{\vec{m}} 
	\frac{(|\vec{m}_{>0}|)!}{(L-|\vec{m}_{>0}|)!}
	\prod_{k=1}^L \big|\tr(K^{(m_k)}_{s,t})\big|.
\end{align}
This is exactly~\eqref{eq-interchange}.

The bounds from Lemma~\ref{l.trace} \revi{ensure} that $ \prod_{k=1}^L |\tr(K^{(m_k)}_{s,t})| \leq \Con(t,n)^L $, for all $ s\in[0,1]$. 
Given this, it is straightforward to verify that, when summed over $ L\geq 1 $,
the r.h.s.\ of \eqref{e.K^n.diff..} converges uniformly over  $ [0,1]\ni s $, for fixed $ t,n $.
Proposition~\ref{p.termwise} applied with $ f_L(s) = (-1)^L\tr(K_{s,t}^{\wedge L}) $ completes the proof.
\end{proof}

\subsection{Bounds.}\label{last-piece}
The goal of this subsection is to bound the term $ \calB_{p,L}(t) $, defined in~\eqref{e.calBL}.
Recall $ \Msp(L,n) $ from~\eqref{e.Msp}.
Referring to~\eqref{e.calBL} and~\eqref{eq-interchange}, we see that
\begin{align}
	\label{e.calBL.id}
	\revi{\left|\calB_{p,L}(t)\right|} 
	\leq
	\frac{1}{\Gamma(1-\alpha)}
	\sum_{\vec{m}\in\Msp(L,n)}\binom{n}{\vec{m}} 
	\frac{(|\vec{m}_{>0}|)!}{(L-|\vec{m}_{>0}|)!}
	\prod_{k=1}^L \int_0^1 \, s^{-\alpha} \, \big|\tr(K^{(m_k)}_{s,t})\big| \, \d s.	
\end{align}
In view of~\eqref{e.calBL.id}, we \revi{first} establish
\begin{proposition}\label{p.error.bd.}
Fix any $ t_0,p_0> 0 $. There exists a constant $\Con=\Con(t_0,p_0)>0$ such that for all $t>t_0$, $p \ge p_0$, $ L\geq 2 $, and $ \vec{m}=(m_1,\ldots,m_L)\in\Msp(L,n) $,
\begin{align}\label{lap-eq-def}
	\frac1{\Gamma(1-\alpha)}\int_0^1 s^{-\alpha} \prod_{j=1}^{\revi{L}} |\tr(K_{s,t}^{(m_j)})|\d s
	\leq	
	n\cdot n!\,\Con^L\, \,t^\frac{1}{2} \, e^{\frac{p^3t}{12}-\kappa_p t }.
\end{align}
where $n:=\lfloor p \rfloor+1$ and $\alpha:=p-\lfloor p\rfloor$ and $\kappa_p:=\min\{\frac{1}6,\frac{p^3}{16}\}.$
\end{proposition}

\begin{proof}
Fix $ L\geq 2 $, $ p\geq p_0 $, $ \vec{m}=(m_1,\ldots,m_L)\in\Msp(L,n) $. 
To simplify notation, throughout this proof we assume $ t\geq t_0 $ and $ p\geq p_0 $, and write $ \Con=\Con(t_0,p_0) $. Set 
\begin{align}
	\label{e.calI}
	\calI:=\frac1{\Gamma(1-\alpha)}\int_0^1 s^{-\alpha} \prod_{j=1}^{\revi{L}} |\tr(K_{s,t}^{(m_j)})|\d s
\end{align}
and $ |\vec{m}_{>0}|:=r $. Assume without loss of generality $ 0< m_1,\ldots,m_{r} $ and $ m_{r+1}=\cdots=m_L=0 $. 
Our goal is to bound $ \calI $.
In~\eqref{e.calI}, perform a change of variable $s=e^{-t\s}$, apply the bounds from Lemma~\ref{l.trace},
and recall $ \expf_q $ from~\eqref{e.expq}.
We have, for all $ t\ge t_0 $, 
\begin{align}\label{Msig.}
	\calI 
	\le  
	\frac{ \Con^L}{\Gamma(1-\alpha)}
	\int_0^{\infty} e^{t\s\alpha} \Big( \Con e^{t\expf_1(\min\{\sqrt{\s},\frac12 \}-t\s)} \Big)^{L-r} \,\cdot \prod_{j=1}^r (m_j)! e^{t \expf_{m_j}(\min\{\sqrt{\s},\frac{m_j}2\})}  \cdot\, te^{-t\s}\d\s,
\end{align}
Given that $ m_1+\ldots+m_L=n$ we have $ \prod_{j=1}^r (m_j)! \le n! $.
Apply this bound in~\eqref{Msig.}, and combine the exponential functions in the \revi{integrand} together to get $ \exp(tM(\s)) $, where
\begin{align}
	\label{e.Ms}
	M(\s):=(\alpha-L+r-1)\s+(L-r)\expf_1(\min\{\sqrt{\s},\tfrac12 \})+\sum_{j=1}^r \expf_{m_j}(\min\{\sqrt{\s},\tfrac{m_j}2\}).
\end{align}
We arrive at
\begin{align}\label{Msig}
	\calI 
	\le  
	\frac{t \Con^L n!}{\Gamma(1-\alpha)}  \int_0^{\infty} e^{tM(\s)}\,\d\s.
\end{align}

Our next step is to bound the exponent $ M(\s) $, which we do in several different cases.
\begin{enumerate}[leftmargin=0pt]
\item \textbf{When $ \sigma\in[0,\frac14] $.}\\%
	Recall from~\eqref{e.expq} that $ \expf_q(x) $ is increasing on $ x\in[0,q/2] $.
	Hence, for $ \sigma\leq \frac14 $, the `min' operators in \eqref{e.Ms} always pick up $ \sqrt{\s} $, whence $ M(\s) $ simplifies into
	$ M(\s)=p\s-\tfrac{4L}3\s^{3/2}:=g_1(\s). $
	This function $ g_1 $ achieves its maximum $\frac{p^3}{12L^2}$ at $\s=\frac{p^2}{4L^2}$. Further,
	$
		g_1(\s)-\frac{p^3}{12L^2}=-\tfrac13(\sqrt{\s}-\frac{p}{2L})^2(p+4L\sqrt{\s}) \le -\frac{p}3(\sqrt{\s}-\frac{p}{2L})^2.
	$
	This gives
	\begin{align}\label{e.case1.}
		M(\s) \leq \tfrac{p^3}{12L^2}-\tfrac{p}{3}(\sqrt{\s}-\tfrac{p}{2L})^2.
	\end{align}
\item \textbf{When $r\ge 2$ and $ \sigma\in(\frac14,\infty) $.}\\%
	In this case, referring to~\eqref{e.expq}, we see $ \expf_1(\min\{\sqrt\s,\frac12\}) = \expf_1(\frac12) = \frac1{12} $.
	Hence $ M(\s) $ simplifies into
	$
		M(\s) = \s(\alpha-1)-(L-r)(\s-\frac1{12})+\sum_{j=1}^r\expf_{m_j}(\min\{\sqrt{\s},\frac{1}2m_j\}).
	$
	Forgo the negative term $ -(L-r)(\s-\frac1{12}) $ and use~\eqref{e.expq.property} to bound $ \expf_{m_j}(\min\{\sqrt{\s},\frac{1}2m_j\}) \leq \frac{1}{12}m_{j}^3 $.
	We have
	\begin{align}\label{est2.2.}
		M(\s) \le  \s(\alpha-1)+\sum\nolimits_{j=1}^r\tfrac{1}{12}m_j^3.
	\end{align}
	Recall that $ m_1+\ldots+m_r=n $.
	The cubic sum in~\eqref{est2.2.} tends to be larger when mass concentrates on fewer $ m_i $'s.
	Under the \revi{current} assumption $ r\geq 2 $, it is conceivable that the cubit sum is at most $ (n-1)^3+1^3 $.
	To prove this, write $ m_1^3+\ldots+m_n^3 \le m_1^3+(m_2+\cdots+m_n)^3=m_1^3+(n-m_1)^3 $,
	and note that the last expression, as a function of $ m_1\in[1,n-1] $, reaches its maximum at $ m_1=1,(n-1) $. 
	\revi{Using this bound on the cubic sum we have
	\begin{align}\label{e.case2.}
		M(\s) \le \s(\alpha-1)+\tfrac{1}{12}((n-1)^3+1).
	\end{align}}
		
\item \textbf{When $r=1$ and $ \sigma\in(\frac{n^2}{4},\infty) $.}\\%
	Under current assumptions, using~\eqref{e.expq.property} we see that
	\begin{align}\label{e.case3.}
		M(\s)=\s(\alpha-L)+\frac{n^3+L-1}{12}.
	\end{align}			
\item \textbf{When $r=1$, $ \sigma\in(\frac14,\frac{n^2}{4}] $, and $ p>L $.}\\%
	When $ r=1 $ and $ \s>\frac14 $, the exponent $ M(\s) $ takes the form
	\begin{align}
		\label{e.case2.M}
		M(\s)
		=
		\s(\alpha-L)+\tfrac{1}{12}(L-1)+\expf_n(\sqrt\s)
		=
		\s(n+\alpha-L)+\tfrac{1}{12}(L-1)-\tfrac43\sigma^{3/2}
		=
		:g_2(\sigma).
	\end{align}
	Differentiating in $ \sigma $ shows that $ g_2 $ reaches its maximum $ \frac{1}{12}(p-L+1)^3+\frac{1}{12}(L-1) $ at $ \s=(p-L+1)^2/4 $.
	Further
	$
		g_2(\s)-\tfrac{(p-L+1)^3}{12}-\tfrac{L-1}{12} = -\frac13(\sqrt\s-\tfrac{p-L+1}2)^2(p-L+1+4\sqrt\s).
	$
	Using the current assumption $ p>L $ to bound $ (p-L+1+4\sqrt\s) \geq 1 $ we get
	\begin{align}\label{e.case4.}
		M(\s) \leq \tfrac{(p-L+1)^3}{12}+\tfrac{L-1}{12}-\tfrac13(\sqrt\s-\tfrac{p-L+1}2)^2.
	\end{align}	
\item \textbf{When $r=1$, $ \sigma\in(\frac14,\frac{n^2}{4}] $, and $ p\leq L $.}\\%
	Here we also have the expression~\eqref{e.case2.M} of $ M(\s) $.
	Under the current assumption $ p\leq L $.
	Differentiating in $ \s $ shows that $ g_2 $ is decreasing on $ s\in(\frac14,\frac{n^2}{4}] $.
	Further $ g_2(\s)-g_2(\frac14)  = (p-L+1)(\s-\frac14)-\frac43(\s^{3/2}-\frac18) $.
	Use the current assumptions to bound $ (p-L+1)(\s-\frac14) \leq (\s-\frac14) $.
	We get
	$ g_2(\s)-g_2(\frac14) \leq (\s-\frac14)-\frac43(\s^{3/2}-\frac18) = -\frac13(1+4\sqrt\s)(\sqrt\s-\frac12)^2 $.
	Further bound $ -\frac13(1+4\sqrt\s) \leq 1 $.
	Together with $g_2(\frac14)=\frac{3p-2L}{12}$, we have
	\begin{align}\label{e.case5.}
		M(\s) \leq \tfrac{1}{12}(3p-2L)-(\sqrt\s-\tfrac12)^2.
	\end{align}
\end{enumerate}

Now, in each of the preceding case, 
use the \revi{respective} bound~\eqref{e.case1.}, \eqref{e.case2.}, \eqref{e.case3.}, \eqref{e.case4.}, or \eqref{e.case5.}
to bound the integral $ \int_A e^{tM(\s)} \d \s $ on the relevant range $ A $.
For the resulting integral,
\begin{enumerate}[leftmargin=0pt]
\item perform a change of variable $ \sqrt\s\mapsto u $, which introduces a factor $ 2u $;
	bound this factor by $ 2\cdot\frac12 $, release the range of integration from $ u\in(0,\frac12) $ to $ u\in\R $,
	and evaluate the resulting integral.
\item \revi{evaluate the resulting integral.}
\item evaluate the resulting integral.
\item perform a change of variable $ \sqrt\s\mapsto u $, which introduces a factor $ 2u $;
	bound this factor by $ 2u\leq n $,
	release the range of integration from $ u\in(\frac12,\frac{n}{2}) $ to $ u\in\R $, and evaluate the resulting integral.
\item perform a change of variable $ \sqrt\s\mapsto u+\frac12 $, which introduces a factor $ 2u+1 $;
	release the range of integration from $ u\in(0,\frac{n-1}{2}) $ to $ u\in\R_{\geq 0} $, and evaluate the resulting integral.
\end{enumerate}
\medskip
This gives the following bound on $ \int_A e^{tM(\s)} \d \s $ on the relevant region $ A $:
\begin{enumerate}
\item\label{enu.1} $ \Con\,p^{-\frac12}t^{-\frac12} \exp(t\frac{p^3}{12L^2})) $
\revi{\item\label{enu.2} $ \Con\, t^{-1}(1-\alpha)^{-1} \exp( t(\frac{(n-1)^3+1}{12}-\frac{1-\alpha}{4})) $}
\item\label{enu.3} $ \Con\, t^{-1}(L-\alpha)^{-1} \exp(t(\frac{(n^3+L-1)}{12}-\frac{n^2(L-\alpha)}4) ) $
\item\label{enu.4} $ \Con\, t^{-\frac12} n \exp(\frac{t}{12}((p-L+1)^3+(L-1))) $
\item\label{enu.5} $ \Con\,(t^{-1}+t^{-1/2}) \exp(\frac{t}{12}(3p-2L)) $
\end{enumerate}
\noindent{}Our goal is to have the exponent strictly less that $ t\frac{p^3}{12} $.
\begin{enumerate}
\item[(1)] Since $ L\geq 2 $ we have $ \frac{p^3}{12L^2} \leq \frac{p^3t}{12}-\frac{p^3}{16} $.
\revi{\item[(2)] Under the current assumption $ r\geq 2 $ forces $ n\geq 2 $, and $p\ge 1$ and hence 
\begin{align*}
\tfrac{(n-1)^3+1}{12}-\tfrac{1-\alpha}{4}=\tfrac{p^3}{12}-\tfrac{(p(n-1)-1)\alpha}{4}-\tfrac{\alpha^3}{12}-\tfrac16 \leq	 \tfrac{p^3}{12}-\tfrac16.
\end{align*}}	
\item[(3)] The exponent in~\eqref{enu.3} therein is decreasing in $ L $. This gives
	\begin{align*}
		\tfrac{n^3+L-1}{12}-\tfrac{n^2(L-\alpha)}{4} \le \tfrac{n^3+1}{12}-\tfrac{n^2(2-\alpha)}{4} =\tfrac{p^3}{12}-\tfrac{(1-\alpha)^2(p+2n)}{12}-\tfrac{3n^2-1}{12}\le \tfrac{p^3}{12}-\tfrac16.
	\end{align*} 
\item[(4)] View the exponent in~\eqref{enu.4} as a function $ g_3(x) := \frac{1}{12}((p-x)^3+x) $ of $ x:= L-1 $.
	Under the relevant assumption $ p\geq L $ and $ 2\leq L $, differentiating $ g_3 $ show that $ g_3 $ is maximized at $ x=1 $.
	This gives $ \frac{1}{12}((p-L+1)^3+(L-1)) \leq g_3(1)= \frac{1}{12}(p^3-3p^2+3p) \le \frac1{12}(p^3-6) $.
\item[(5)]  Use $ L\geq 2 $ to bound $ \frac{1}{12}(3p-2L) \leq \frac{1}{12}(3p-4) $.
	For $ p\geq 0 $, the last expression is always bounded by $ \frac{p^3}{12}-\frac{1}{6} $, which gives
	$ \frac{1}{12}(3p-2L) \leq \frac{p^3}{12}-\frac{1}{6} $.
\end{enumerate}

Collect the preceding discussion and refer back to~\eqref{Msig}. We arrive at
\begin{align*}
	\calI \leq e^{\frac{p^3t}{12}} \Con^L \frac{n!}{\Gamma(1-\alpha)}
	\Big( 
		p^{-\frac12}t^\frac12 e^{-\frac{p^3t}{16}} 
		+ \frac{e^{-\frac{t}{6}}}{(1-\alpha)}
		+ e^{-\frac{t}{6}}		
		+ n t^\frac12 e^{-\frac{t}{2}}
		+ (1+t^\frac12) e^{-\frac{t}{6}}
	\Big).
\end{align*}
Further apply the bounds $ p^{-\frac12}\leq p_0^{-\frac12}=\Con $, $ \frac{1}{\Gamma(1-\alpha)} \leq \Con $, and $ \frac{1}{(1-\alpha)\Gamma(1-\alpha)} \leq \Con $, for all $ \alpha\in[0,1) $. We conclude the desired result.
\end{proof}

\begin{proposition} \label{p.error.bd} Fix any $ t_0,p_0> 0 $. 
Recall $ \calB_{p,L}(t) $ from~\eqref{e.calBL}.
There exists a constant $\Con=\Con(t_0,p_0)>0$ such that for all $t>t_0$ and $p \ge p_0$,
\begin{align}\label{ho-eq}
	\sum_{L\geq 2} |\calB_{p,L}(t)|
	\le  
	n\cdot (n!)^2 \, (n\Con)^n\,t^{\frac12} \, e^{\frac{p^3t}{12}-\kappa_pt},
\end{align}
where $n:=\lfloor p \rfloor+1$ and $\alpha:=p-\lfloor p\rfloor$, and $\kappa_p:=\min\{\frac{1}6,\frac{p^3}{16}\}.$
\end{proposition}
\begin{proof}
Multiply both sides of~\eqref{eq-interchange} by $ s^{-\alpha} $, integrate the result over $ s\in[0,1] $,
and apply the bound~\eqref{lap-eq-def}. We get, for $ \Con=\Con(t_0,p_0) $,
\begin{align*}
	\text{l.h.s.\ of }\eqref{ho-eq}
	\leq
	(n+1)! \,t^{\frac12}\, e^{\frac{p^3t}{12}-\kappa_pt} 
	\sum_{L \geq 2}\sum_{\vec{m}\in\Msp(L,n)}\binom{n}{\vec{m}}\frac{({|\vec{m}_{>0}|})!\Con^L}{(L-{|\vec{m}_{>0}|})!}
\end{align*}
Within the last expression,
use $ |\vec{m}_{>0}| \leq n $ to bound $ \frac{({|\vec{m}_{>0}|})!}{(L-{|\vec{m}_{>0}|})!} \leq \frac{n!}{((L-n)_+)!} $,
and evaluate the sum $ \sum_{\vec{m}\in\Msp(L,n)}\binom{n}{\vec{m}} = L^n $.
This gives
\begin{align}
	\label{ho-eq.1}
	\text{l.h.s.\ of }\eqref{ho-eq}
	\leq
	n\cdot\,(n!)^2 \,t^{\frac12}\, e^{\frac{p^3t}{12}-\kappa_pt} 
	\sum_{L \geq 2}\frac{L^n\Con^L}{((L-n)_+)!}.
\end{align}
\revi{%
In the sum in \eqref{ho-eq.1}, bound $ L^n \leq (2n+(L-2n)_+)^n \leq 2^n(2n)^n + 2^n((L-2n)_+)^n $, use $ \frac{((L-2n)_+)^n}{((L-n)_+)!} \leq \frac{1}{((L-2n)_+)!} $, and evaluate the resulting series. The result shows that the sum in \eqref{ho-eq.1} is bounded by $ (n\Con)^n $. This completes the proof.}
\end{proof}

\section{Proof of Theorem~\ref{thm.main} and Theorem~\ref{thm.main.}} \label{sec.pfthm}
We begin with the proof of Theorem~\ref{thm.main.}.
Lemma~\ref{l.diff.term} \revi{justifies} the passing of derivatives in \eqref{e.derived}.
Recall the definition of $ \til\calA_p(t) $, $ \calB_{p,L}(t) $, $ L\geq 2 $, and $ \calB_{p,1}(t) $ from in~\eqref{e.calA}, \eqref{e.calBL}, and \eqref{e.calBo},
we have
$
	\Ex[(\calZ(2t,0)e^{\frac{t}{12}})^p]
	=
	\til\calA_{p}(t) + \sum_{L \geq 1} \calB_{p,L}(t).
$
Further, recall from \eqref{e.calA.decomp} that $ \til\calA_{p}(t)=\calA_p(t)-\hat\calA_p(t) $, so
\begin{align*}
	\Ex\big[(\calZ(2t,0)e^{\frac{t}{12}})^p\big]
	=
	\calA_{p}(t) - \hat\calA_{p}(t) + \sum_{L \geq 1} \calB_{p,L}(t).
\end{align*}
Given the bound~\eqref{e.calB0bd} and the bounds from Propositions~\ref{trace-estimate} and \ref{p.error.bd},
Theorem~\ref{thm.main.} now follows for $ \calB_{p}(t) := - \hat\calA_{p}(t) + \sum_{L \geq 1} \calB_{p,L}(t) $.

Next, Theorem~\ref{thm.main}\ref{thm.main.mom} follows immediately from Theorem~\ref{thm.main.}.
It now remains only to show Theorem~\ref{thm.main}\ref{thm.main.ldp}.
We will establish the large deviation upper and lower bound separately.
To simplify notation set $ V_t := \calH(2t,0)+\tfrac{t}{12} $.
Fix $ y>0 $. \revi{Markov's} inequality gives $	\Pr[ V_t \ge ty] \le e^{-p y}\Ex[e^{p V_t}] $.
Apply Theorem~\ref{thm.main}\ref{thm.main.mom}, take logarithm, and divide by $ t $. We obtain, for all $ p>0 $,
\begin{align} \label{upb1}
	\limsup_{t\to\infty} \tfrac1t\log \Pr\big[ V_t \ge ty \big] 
	\le 
	-p y + \tfrac{1}{12} p^3.
\end{align} 
Minimizing the right side of \eqref{upb1} over $p>0$, we obtain the desired large deviation upper bound
\begin{align*} 
	\limsup_{t\to\infty} \tfrac1t\log \Pr[V_t \ge ty\big] \le -\tfrac43y^{3/2}.
\end{align*} 

For lower bound we employ the standard change-of-measure argument and utilize the strict convexity of the function $ \frac{1}{12}p^3 $, $ p>0 $.
Fix $ \e>0 $, set $ q_*:=2(y+\e)^{1/2} $, and let $ \til{V}_{t} $ denote the random variable with the tilted law 
$ \Pr[\til{V}_t\in A] = \frac{1}{\Ex[ e^{q_*V_t} ]} \Ex[ e^{q_*V_t} \ind_\set{A}(V_t) ] $.
We write
\begin{align}
	\label{e.com}
	\Pr[V_t \geq ty ] = \Ex[ e^{-q_*\til{V}_t} \ind_\set{\til{V}_{t} \geq ty} ] \cdot \Ex[ e^{q_*V_t} ]
	\geq
	e^{-tq_*(y+2\e)} \Ex[ e^{q_*V_t} ] \, \Pr\big[ \til{V}_{t} \in [ty,t(y+2\e)] \big].
\end{align}
Our goal is to show that $ \lim_{t\to\infty}\Pr[ \til{V}_{t} \in [ty,t(y+2\e)] ]=1 $.
To this end, for $ \revi{\lambda\in(0,q_*) }$ bound the complement probability by \revi{Markov's} inequality as
\begin{align*}
	\Pr\big[\til{V}_t<ty\big] \leq e^{\lambda ty} \Ex[e^{-\lambda \til{V}_t}]
	&=
	e^{\lambda ty} \frac{ \Ex[e^{(q_*-\lambda)V_t}] }{ \Ex[e^{q_*V_t}] },
\\
	\Pr\big[\til{V}_t>t(y+2\e)\big] \leq e^{-\lambda t(y+2\e)} \Ex[e^{\lambda \til{V}_t}]
	&=
	e^{-\lambda t(y+2\e)} \frac{ \Ex[e^{(q_*+\lambda)V_t}] }{ \Ex[e^{q_*V_t}] }.
\end{align*}
Take log, divide the result by $ t $, and apply Theorem~\ref{thm.main}\ref{thm.main.mom}.
We obtain
\begin{align}
	\label{e.pto0}
	\limsup_{t\to\infty}
	\tfrac{1}{t}\log \Pr\big[\til{V}_t<ty\big] 
	&\leq
	y\lambda + \tfrac{1}{12}(q_*-\lambda)^3 - \tfrac{1}{12}q_*^3,
\\
	\label{e.pto1}
	\limsup_{t\to\infty}
	\tfrac{1}{t}\log \Pr\big[\til{V}_t>t(y+2\e)\big] 
	&\leq
	-(y+2\e)\lambda + \tfrac{1}{12}(q_*+\lambda)^3 - \tfrac{1}{12}q_*^3.
\end{align}
Now, view the r.h.s.\ of~\eqref{e.pto0} and \eqref{e.pto1} as functions of $\revi{ \lambda \in (-q_*,q_*) }$.
It is readily checked that these functions are strictly convex, zero at $ \lambda=0 $, and \revi{has} negative derivative at $ \lambda=0 $.
Hence there exists a small enough $ \lambda_*=\lambda_*(\e,y)>0 $ such that the r.h.s.\ of~\eqref{e.pto0} and \eqref{e.pto1} are negative for $ \lambda=\lambda_* $.
This gives  $ \lim_{t\to\infty}\Pr[ \til{V}_{t} \in [ty,t(y+2\e)] ]=1 $.
Use this in~\eqref{e.com}, take log, divide the result by $ t $, and apply Theorem~\ref{thm.main}\ref{thm.main.mom} to get
\begin{align*} 
	\liminf_{t\to\infty} \tfrac1t\log \Pr[V_t \ge ty\big] \ge -q_*(y+2\e) + \tfrac{1}{12}q_*^3 = - \tfrac{4}{3}(y+\e)^{3/2}-2\e(y+\e)^{1/2}.
\end{align*} 
Since $ \e>0 $ was arbitrary, sending $ \e\to 0 $ gives the desired large deviation lower bound.

\bibliographystyle{abbrv}		
\bibliography{frt}
\end{document}